\pdfoutput=1

\documentclass[leqno,10pt]{amsart}

\usepackage{amsmath}

\usepackage{amssymb}
\usepackage{graphicx}

\newtheorem{thm}{Theorem}

\newtheorem{prop}[thm]{Proposition}
\newtheorem{lemma}[thm]{Lemma}
\newtheorem{cor}[thm]{Corollary}

\theoremstyle{definition}
\newtheorem{defn}{Definition}
\newtheorem{ex}{Example}

\theoremstyle{remark}
\newtheorem{remark}{Remark}

\numberwithin{equation}{section}

\def\R{\mathbb{R}}

\def\Z{\mathbb{Z}}

\begin{document}
\title[]{Topologically embedded pseudospherical cylinders}

\author{Emilio Musso}
\address{(E. Musso) Dipartimento di Scienze Matematiche, Politecnico di Torino,
Corso Duca degli Abruzzi 24, I-10129 Torino, Italy}
\email{emilio.musso@polito.it}

\author{Lorenzo Nicolodi}
\address{(L. Nicolodi) Di\-par\-ti\-men\-to di Scienze Ma\-te\-ma\-ti\-che, Fisiche e Informatiche,
Uni\-ver\-si\-t\`a di Parma, Parco Area delle Scienze 53/A, I-43124 Parma, Italy}
\email{lorenzo.nicolodi@unipr.it}

\thanks{Authors partially supported by MIUR (Italy) under the PRIN project
\textit{Variet\`a reali e complesse: geometria, topologia e analisi armonica};
and by the GNSAGA of INDAM}





\begin{abstract}
The class of traveling wave solutions of the
sine-Gordon equation is known to be in 1-1 correspondence with the class of (necessarily singular)
pseudospherical surfaces in Euclidean space with screw-motion symmetry: the pseudospherical helicoids.
We
explicitly describe all pseudospherical helicoids
in terms of elliptic functions.
This solves a problem posed by Popov \cite{Po}.
As an application, countably many
continuous families
of topologically embedded pseudospherical helicoids are constructed.
%
%
A (singular) pseudospherical helicoid
is proved to be either
a dense subset of a region bounded by two coaxial cylinders, a topologically
immersed cylinder with helical self-intersections, or a topologically embedded cylinder
with helical singularities, called for short a pseudospherical twisted column.
%
Pseudospherical twisted columns
are characterized
by four phenomenological invariants: the helicity $\eta\in \mathbb{Z}_2$,
the parity $\epsilon\in \Z_2$, the wave number $\mathfrak n\in \mathbb{N}$, and the aspect
ratio $\mathfrak{d}>0$, up to translations along the screw axis.
A systematic
procedure for explicitly
determining all pseudospherical twisted columns from the invariants is provided.
\end{abstract}

\maketitle

\section{Introduction}\label{s:intro}

The study of {\em pseudospherical surfaces}, i.e., surfaces in $\R^3$ with constant Gauss curvature $K = -1$,
is a classical subject in differential geometry dating back to the second half of the 19th century \cite{Bi,Ei}.
The renewed interest in the subject is
principally due to the fact that pseudospherical surfaces constitute an integrable system
governed by the sine-Gordon equation
\begin{equation}\label{sge}
\phi_{ss} -\phi_{tt}=\sin \phi
\end{equation}
\cite{BPT,CaliniIvey,Olverbook,PaBAMS,RSBook,Sa,Te,TU}.
Equation \eqref{sge}
amounts to the integrability conditions (Gauss--Codazzi equations)
for the linear system of Gauss--Weingarten obeyed by the
tangent frame to a pseudospherical surface with respect to curvature line coordinates.
It is well-known that there is a one-to-one
correspondence between local solutions $\phi$ of the sine-Gordon equation
with $0<\phi < \pi$ and local pseudospherical surfaces up to rigid motion.
Although the sine-Gordon
equation has many global solutions defined on the whole $\R^2$, the corresponding surfaces always have
singularities. In fact,
according to the proof of Hilbert's theorem asserting that there is no complete
immersed pseudospherical surface in $\R^3$,
any smooth solution $\phi : \R^2 \to \R$ of the sine-Gordon
equation attains values that are multiples of $\pi$ \cite{Hilbert,PoMa}. This implies
that the map $f_\phi : \R^2 \to \R^3$ corresponding to $\phi$ is smooth but fails to be an immersion at points $p$
where $\sin \phi(p) =0$, i.e., when
$\phi = k\pi$, for some integer $k$.
At these points
the map $f_\phi$
has rank 1.
It then follows that there is a one-to-one correspondence between smooth solutions $\phi : \R^2 \to \R$
of the sine-Gordon equation and smooth maps $f_\phi : \R^2 \to \R^3$ such that
(1) $\text{rank} f_\phi \geq1$ everywhere; and (2)
if $\text{rank} f_\phi =2$ on an open set $U$ of $\R^2$, then ${f_\phi}_{|U}$ is a pseudospherical
immersion (see \cite{PTbook,Po,RSBook,TU} and the references therein).

\vskip0.1cm


In his 2014 monograph \cite[Chapter 3]{Po}, Andrey Popov
posed the problem of
the explicit
description and computation of the (necessarily singular) pseudospherical surfaces corresponding to
the stationary traveling wave solutions (single-phase solutions) of the sine-Gordon equation
\cite{Bi,RMcL,Ov,Po,RSBook}.
Such surfaces are known to be the pseudospherical surfaces with screw-motion symmetry:
the {\em pseudospherical helicoids} \cite{Bi, Ei}.
Examples of pseudospherical helicoids include the well-known Dini helicoids
which contain the pseudosphere and
correspond
to the 1-soliton solutions of the sine-Gordon equation.
%
%
In this paper, we explicitly determine all pseudospherical helicoids
and investigate their global geometry. This leads, in particular, to
the discovery of
countably many continuous families of topologically embedded examples of singular
pseudospherical helicoids, namely topologically embedded pseudospherical cylinders.
%
Our study is related to recent work on pseudospherical surfaces with
singularities \cite{Brander,Po}
and modern applications of the geometry of pseudospherical helicoids
to the study of modulated wave solutions of certain mathematical models
in nonlinear elasticity \cite{RS1,RSW,RSZ}.
%



\vskip0,1cm
While the local geometry of pseudospherical helicoids was
well understood already by the end of the 19th century \cite{Bi},
little is known about their global geometry (see \cite{Po} for an updated overview).
It is our purpose to address natural global questions such as
(1) find explicit expression for the (necessarily singular) map $f_\phi$
corresponding to a stationary traveling wave solution $\phi$ of the sine-Gordon equation;
and (2) construct examples of singular pseudospherical helicoids
which are topologically embedded cylinders,
other than
the pseudosphere
and the rotation pseudospherical surfaces of cylindrical type.
Starting from a single-phase solution $\phi$ to the sine-Gordon equation,
we will explicitly compute,
in terms of elliptic functions,
the map $f_\phi$ that provides
a global parametrization of the pseudospherical helicoid corresponding to $\phi$.
With this in hand, we will prove the existence and construct continuous families of
pseudospherical helicoids which are topologically embedded cylinders
and have helical singularities.
Although rather simple from a conceptual point of view, the above geometric problems
present a certain computational difficulty due to the presence of the singularities
and the fact that
the single-phase solutions of the sine-Gordon equation are in
general transcendental elliptic functions. This requires an additional detailed
study of the properties of Jacobian functions
which is an interesting topic on its own
\cite[\S3.3.3]{Po}.
The theoretical approach to the problem is based on the method of moving frames
\cite{Olver1,Olver2,JMNbook}.

As for constant mean curvature (CMC) surfaces, we mention that
helicoidal CMC surfaces were studied by do Carmo and
Dajczer \cite{dCaDa} and further by Roussos \cite{Roussos} and Perdomo \cite{Perdomo}, among others.
Such surfaces, and in particular helicoidal CMC cylinders,
were also studied by Burstall and Kilian \cite{BuKi} using
methods from the theory of equivariant harmonic maps and integrable systems.
Notice that non-rotational helicoidal CMC surfaces are always nonsingular
and
have self-intersections along helices, which implies they can never be embedded.

\subsection{Description of results}
We will now briefly describe our main results.
A {\it helicoid}, or a {\it helicoidal surface}, is a surface $S$ in $\R^3$,
possibly with singularities, swept out by an appropriately chosen planar curve $\Gamma$
animated by a helicoidal motion around an oriented line.
The orbits of this motion (helices) through the initial profile curve foliate the
helicoid.
A helicoid $S$ is of {\it translational kind} if its profile curve $\Gamma$
is simple,
has a symmetry group consisting of
a nontrivial group of translations along
the axis of the helicoidal motion,
%
%
and has only ordinary cusps as possible singular points (cf. Figures \ref{FIG4} and \ref{FIG6}).
%
Associated with a translational helicoid there are four main phenomenological invariants:
the {\it helicity} $\eta\in \Z_2$;
the {\it parity} $\epsilon \in \Z_2$; the {\it wave number} $\mathfrak{n}\in \R^+$; and the
{\it aspect ratio} $\mathfrak{d}\in \R^+$ (cf. Section \ref{s:pre} for the precise definitions).\footnote{Here
$\mathbb Z_2 = \{\pm 1\}$ and $\R^+$ stands for
the set of positive reals.}
Depending on whether the wave number $\mathfrak{n}$ is an irrational, a rational, or a
natural number, a translational helicoid $S$ is a dense subset of a region of $\R^3$ bounded by two coaxial cylinders,
a topologically immersed cylinder with helical self-intersections, or a topologically embedded
cylinder, that is, a closed subset of $\R^3$ homeomorphic to $S^1\times \R$.
The helicoids of the last class are called {\it twisted columns}.
The first main result of the paper will be the following.

\begin{thm}\label{thm:main-1}
All pseudospherical helicoids are of translational kind.
\end{thm}

Our second main result characterizes pseudo\-spherical twisted columns
in terms of the four phenomenological invariants and
provides an effective method for explicitly constructing all of them. We will prove the
following.

\begin{thm}\label{thm:main}
For given $(\eta,\epsilon,\mathfrak{n},\mathfrak{d})$ $\in$
$\mathbb{Z}_2\times \mathbb{Z}_2\times \mathbb{N}\times \R^+$,
there exists a unique congruence class of pseudospherical twisted columns
with
{\em helicity} $\eta$, {\em parity} $\epsilon$, {\em wave number}
 $\mathfrak{n}$, and {\em aspect ratio} $\mathfrak{d}$.
\end{thm}

The proof of Theorem \ref{thm:main} is constructive
and can be used to actually compute explicit parametrizations
of pseudospherical twisted columns in terms of the phenomenological invariants
with the help of
any symbolic manipulation program supporting
%
elliptic functions and integrals.
The explicit construction
also requires some numerics for the determination of
the zeros of an analytic function (cf. Section \ref{s:proof-mthm}).\footnote{The proof of the theorem can also
be used to find all
linear Weingarten twisted columns of hyperbolic type, that is, the twisted columns such that
$\alpha K + \beta H = 1$, where $\beta\neq 0$ and $\alpha$ are two constants
satisfying $\alpha+\beta^2/4 <0$ and $K$ and $H$ are the Gauss and mean curvature, respectively.}


\vskip0.1cm

The paper is organized as follows.
Section \ref{s:pre} recalls
the one-to-one correspondence between local solutions of the
sine-Gordon equation and local pseudospherical surfaces in $\R^3$
and the implicit characterization of pseudospherical helicoids
in terms of stationary traveling wave solutions of the sine-Gordon equation \cite{Bi,PTbook,Po,RS1}.
As a consequence, it follows that the congruence classes of pseudospherical helicoids
only depend on two real parameters $(\mu,r)$ and the helicity $\eta$.
For $\mu=1$, we have the pseudospherical helicoids of Dini, while
for $r=0$ the helicoid reduces to a surface of revolution.
Being the geometry of such surfaces well known \cite{Bi,Ei,Po}, we assume
that $\mu\neq 1$ and $r\neq 0$.
As a first result, we prove that two pseudospherical helicoids
with parameters $(\mu,r)$ and $(\mu,-r)$,
as well as with parameters $(\mu,r)$ and $(\mu/(\mu-1),r)$,
are mirror images of each other (cf. Proposition \ref{R1}).
Thus, with no loss of generality, we may and do restrict our analysis to pseudospherical helicoids with
parameters $\mu >0$, $\mu\neq 1$, and $r>0$.

Section \ref{s:ex-par} computes explicit global parametrizations for pseudospherical
helicoids with parameters $\mu >0$, $\mu\neq 1$, and $r>0$.
%
Following \cite[\S3.3.3]{Po}, this is achieved by
dividing pseudospherical helicoids in two classes: that of
{\em magnetic-type} pseudospherical helicoids, for which $\mu >1$; and that
of {\em electric-type} pseudospherical helicoids, for which $\mu \in (0,1)$.
For either class we compute explicit global parametrizations in terms of Jacobian elliptic
functions and integrals (cf. Theorems \ref{R2} and \ref{R3}).
Operatively, such parametrizations are obtained by applying
the method of moving frames for constructing appropriate lifts to the group of Euclidean motions.
For the general theory of the method of moving frames and some applications we refer to
\cite{Olver1,Olver2,JMNbook}.

Section \ref{s:proof-mthm} proves Theorems \ref{thm:main-1} and \ref{thm:main}.
First, we show that any pseudospherical helicoid is of translational kind.
This is achieved by Lemmas 8 and 9, where
we determine plane profiles for the two classes of pseudospherical helicoids
and compute their phenomenological invariants in terms of the parameters
$(\mu,r)$.
An important consequence of Lemmas 8 and 9 is that any congruence class of pseudospherical
twisted columns has a model pseudospherical twisted column.
This is a typical feature of the application of the method of moving frames
to geometric problems governed by integrable systems
(for similar situations in other geometric contexts, see e.g.
\cite[18--24]{DMNna}).
Next, we prove the existence of pseudospherical
twisted columns with prescribed invariants by showing that certain suitable
modular curves
intersect each other transversally at a single point (cf. Lemmas 10 and 11). Finally,
we show that two pseudospherical twisted columns with different invariants cannot
be congruent (cf. Lemma 12).

\vskip0.1cm
As a basic reference for elliptic functions and
integrals we refer to \cite{La}.
We warn the reader that we use the Jacobi \textit{parameter} $m =k^2$ instead of
the Jacobi \textit{elliptic modulus} $k$.
Symbolic and numerical computations, as well as graphics, are made with the
help of the software package \textsl{Mathematica 11}.

\section{Pseudospherical helicoids: definitions and preliminary results}\label{s:pre}

In this section, after recalling the basic definitions, we prove some preliminary results about
pseudospherical helicoids.

\subsection{Helicoids and twisted columns}\label{ss:helicoids}
Let $\ell$ be an oriented line in $\R^3$ oriented by the unit vector $\vec{\ell}$.
For $\mathfrak{p}>0$ and $\eta\in \Z_2=\{\pm 1\}$, let
$\mathrm{E}^{\ell,\mathfrak{p},\eta}_v : \R^3 \to \R^3$
be the 1-parameter subgroup of rigid motions of $\R^3$
%
%
given by
\[
 \mathrm{E}^{\ell,\mathfrak{p},\eta}_v(\mathbf x) = \mathrm{R}^{\ell}_{2\pi v} (\mathbf x)
 + {\eta} v \mathfrak{p}\vec{\ell},
 \quad v\in \R,
    \]
where $\mathrm{R}^{\ell}_{2\pi v}$ denotes the oriented rotation of an angle $2\pi v$ around $\ell$.
The elements of $\mathrm{E}^{\ell,\mathfrak{p},\eta}_v$
are {\it helicoidal transformations}
around the {\it screw axis} $\ell$ with {\it pitch} $\mathfrak{p}$ and {\it helicity} $\eta$.

\begin{defn}

Let $\widetilde\Gamma\subset \R^3$ be a connected curve, possibly with isolated singularities.
The subset $S\subset\R^3$ defined by
\[
  S=\bigcup\limits_{v\in \R}\mathrm{E}^{\ell,\mathfrak{p},\eta}_v (\widetilde\Gamma)
    \]
is said a {\it helicoid} (or a {\it helicoidal surface}) with
screw axis $\ell$, pitch $\mathfrak{p}$,
helicity $\eta$, and {\it profile} $\widetilde\Gamma$.
The profile curve $\widetilde\Gamma$ generating $S$ is not unique. The helicoid $S$ is right-handed or
left-handed according as
the helicity $\eta = 1$, or $\eta = -1$.
\end{defn}


\begin{defn}
Let $S\subset \R^3$ be a helicoid and let $\mathcal{H}_{\ell}$ be one of the upper half-planes
bounded by the screw axis $\ell$.
The helicoid $S$ is said of {\it translational kind} if there exists a planar profile
$\Gamma\subset\mathcal{H}_{\ell}$ for $S$ satisfying the following conditions:
\begin{enumerate}

\item the full symmetry group of $\Gamma$
is a discrete nontrivial subgroup
$\mathcal{T}\subset \mathbf{E}(3)$ of pure translations along the screw axis $\ell$
(this excludes the possibility of a rectilinear profile);

\item the curve $\Gamma$ is
a simple curve of $\mathcal{H}_{\ell}$ with only ordinary cusps as possible singular points;

\item $\Gamma\cap \ell = \emptyset$.

\end{enumerate}
\end{defn}

If $\mathrm{T}_{\mathfrak{w}\vec{\ell}}: \R^3 \to \R^3$, $\mathbf x\mapsto \mathbf x+\mathfrak{w}\vec{\ell}$,
$\mathfrak{w}>0$, is the generator of the subgroup $\mathcal{T}$, we call $\mathfrak{w}$ the {\it wavelength} and
$\mathfrak{n}=\mathfrak{p}/\mathfrak{w}$
the {\it wave number} of $S$.

\begin{remark}
If the wave number $\mathfrak{n}$ is an irrational number, then $S$ is dense in a region bounded by two coaxial
cylinders. If $\mathfrak{n}$ is a rational number but not an integer, then $S$ is the image of
a topological immersion of a cylinder with self-intersections along a finite number of helices.
If $\mathfrak{n}$ is an integer, then $S$ is a closed subset of $\R^3$ homeomorphic to
$S^1\times \R$.
\end{remark}

\begin{defn}
A translational helicoid with $\mathfrak{n}\in \mathbb{N}$ is said a
{\it twisted column}.
\end{defn}

\begin{defn}
For a helicoid $S$ of translational type,
let $\Gamma_*$ denote the {\em fundamental domain} of
$\Gamma$ with respect to the action of the subgroup $\mathcal{T}$.
Let $\mathfrak{h}$ be the number of singular points in the fundamental domain $\Gamma_*$.
We call  $\epsilon = (-1)^{\mathfrak{h}}\in \Z_2$ the {\it parity} of $S$.
Let $\mathfrak{r}^-$ and $\mathfrak{r}^+$ be the {\it inner} and {\it outer} radii, defined respectively by
\[
 \mathfrak{r}^-=\min\limits_{\mathbf x\in \Gamma}d(\mathbf x,\ell),\quad
 \mathfrak{r}^+=\max\limits_{\mathbf x\in \Gamma}d(\mathbf x,\ell).
 \]
The quotients $\mathfrak{d}^{-}= \mathfrak{w}/\mathfrak{r}^{-}$ and
$\mathfrak{d}^{+}= \mathfrak{w}/\mathfrak{r}^{+}$ are called, respectively, the {\it inner}
and {\it outer aspect ratio}.
\end{defn}

The set $S_*$ of regular points of a helicoid $S$ is open and dense.
We say that $S$ is {\it pseudospherical} if $S_*$ has constant Gaussian curvature $K=-1$.

\subsection{Pseudospherical helicoids and the sine-Gordon equation}

We start by recalling some well-known facts about pseudospherical surfaces of Gauss curvature $K= -1$
and in particular about pseudospherical helicoids
(see e.g. Bianchi \cite[Chapter XV, \S\S245, 250]{Bi},
Popov \cite[Chapters 2 and 3]{Po}).

Let $f : \Omega\subset\R^2 \to \R^3$ is an immersion with constant Gaussian curvature $K = -1$.
Then, for every $p\in \Omega$, there exists an open neighborhood $U\subset \Omega$ of
$p$ and local (curvature line) coordinates $s,t$ on $U$, so that the first and second fundamental forms of $f$ read
\[
 \mathbf{g}=\cos^2\Big(\frac{\phi}{2}\Big)ds^2+\sin^2\Big(\frac{\phi}{2}\Big)dt^2,\quad
   \mathbf{h}=\cos\Big(\frac{\phi}{2}\Big)\sin\Big(\frac{\phi}{2}\Big)(ds^2-dt^2),
  \]
where $\phi : U  \to \R$ (the {\em angular function}) is a solution of the sine-Gordon partial differential equation,
\begin{equation}\label{sge1}
   \phi_{ss}- \phi_{tt}= \sin \phi,
    \end{equation}
satisfying $\phi(U) \subset (0,\pi)$.

Conversely, if $\Omega \subset \R^2$ is simply connected and $\phi = \phi(s,t) : \Omega \to \R$ is a
nonconstant solution of the sine-Gordon equation \eqref{sge1} such that $\phi(\Omega) \subset (0,\pi)$, then
there exists an immersion $f : \Omega\to \R^3$ of constant Gaussian curvature $K = -1$
 with fundamental forms
\[
 \mathbf{g}_{\phi}=\cos^2\Big(\frac{\phi}{2}\Big)ds^2+\sin^2\Big(\frac{\phi}{2}\Big)dt^2,\quad
   \mathbf{h}_{\phi}=\cos\Big(\frac{\phi}{2}\Big)\sin\Big(\frac{\phi}{2}\Big)(ds^2-dt^2).
  \]
The immersion $f$ is implicitly defined by $\mathbf{g}_{f}$ and $\mathbf{h}_{f}$ and
is unique up to rigid motions of $\R^3$.

This latter result have been generalized by Poznyak \cite{Poz} (see also \cite[Theorem 2.7.1]{Po}
and \cite[p. 55]{PTbook})
to the case where the solution $\phi$ is not subject to the condition
that its image is contained in $(0,\pi)$ and has low differentiability.

\begin{thm}
Let $\phi = \phi(s,t):\R^2\to \R$ be a nonconstant $C^4$-solution of the sine-Gordon equation \eqref{sge1}.
Then there exists a $C^2$-map
$f_\phi :\R^2\to \R^3$ whose restriction to the open
set $\{p\in \R^2 \,:\, \phi(p)\neq k\pi,\, k\in \Z\}$ is a pseudospherical immersion with first
and second fundamental forms given by
\[
 \mathbf{g}_{\phi}=\cos^2\Big(\frac{\phi}{2}\Big)ds^2+\sin^2\Big(\frac{\phi}{2}\Big)dt^2
 \quad
 \text{and} \quad
   \mathbf{h}_{\phi}=\cos\Big(\frac{\phi}{2}\Big)\sin\Big(\frac{\phi}{2}\Big)(ds^2-dt^2).
  \]
 \end{thm}

\begin{remark}
The map $f_\phi : \R^2 \to \R^3$ corresponding to a smooth solution $\phi : \R^2 \to \R$
of the sine-Gordon equation \eqref{sge1} is smooth
but fails to be an immersion at points $p$
where
$\phi(p) = k\pi$, for some integer $k$.
At these points the map $f_\phi$ has rank 1.
It follows that there is a one-to-one correspondence between smooth solutions $\phi : \R^2 \to \R$
of the sine-Gordon equation and smooth maps $f_\phi : \R^2 \to \R^3$ such that
(1) $\text{rank} f_\phi \geq1$ everywhere; and (2)
if $\text{rank} f_\phi =2$ on an open set $U$ of $\R^2$, then ${f_\phi}_{|U}$ is a pseudospherical
immersion (see e.g. \cite{PTbook,Po,RSBook,TU}).

\end{remark}

In particular, we have the following (see \cite{Po}).

\begin{cor}
The map
$f_\phi$ parametrizes a pseudospherical helicoid or a
pseudospherical surface of revolution if and only if the angular function $\phi$ is a
{\em stationary traveling wave solution} of \eqref{sge1}, that is, a solution of the form
\[
\phi = \phi(s,t) = \lambda(\xi), \quad \xi= as +b t,
\]
where $\lambda$ is a function of one variable
and $a$, $b$ are real constants with $a\neq 0$, $a^2 \neq b^2$.
We call $\phi$ the {\em potential} of the map $f_\phi$.
\end{cor}

\begin{remark}
The map $f_\phi$ exists globally. However, its explicit expression
seems not to have been computed in the literature (see \cite{Po}).

\end{remark}

If we let $a = \cosh r /L$ and $b = \sinh r /L$, the function $\lambda$ satisfies the simple
pendulum equation $\lambda'' - L^2 \sin \lambda = 0$. Consequently, a stationary traveling wave
solution of the sine-Gordon equation depend on two real parameters $\mu$ and $r$.
For given $\mu$ and $r$, the corresponding traveling wave solution $\phi_{\mu,r}$
of the sine-Gordon equation can be written, up to an affine change of variables, as
\[
 \begin{split}
\phi_{\mu,r}(s,t)&=-2\,\mathrm{am}\big(\frac{\cosh r}{\sqrt{|\mu|}}s+\frac{\sinh r}{\sqrt{|\mu|}}t,\mu\big),
\quad \mu<0\\
\phi_{1,r}(s,t)&=-4\arctan\Big(e^{rs+\sqrt{1+r^2}t}\Big)+\pi,\\
\phi_{\mu,r}(s,t)&=-2\,\mathrm{am}\Big(\frac{\sinh r}{\sqrt{|\mu|}}s+\frac{\cosh r}{\sqrt{|\mu|}}t,\mu\Big),
\quad \mu >0,\mu \neq 1,
\end{split}
 \]
where $\mathrm{am}(-,\mu)$ is the Jacobian amplitude with parameter $\mu$.

\begin{remark}[Jacobian elliptic functions and integrals \cite{La}]
To fix notation, we recall some basic definitions about the Jacobian elliptic
functions and integrals.
If
\[
  u = F(\varphi, \mu) = \int_0^\varphi \frac{d\theta}{(1-\mu \sin^2 \theta)^{\frac{1}{2}}},
   \]
then $\mathrm{am}(u,\mu)  =F^{-1}(u, \mu) = \varphi$ is the Jacobian {\em amplitude}
with {\em parameter} $\mu$.
The {\em Jacobian elliptic functions} $\mathrm{sn}(-,\mu)$, $\mathrm{cn}(-,\mu)$ and
$\mathrm{dn}(-,\mu)$ with
{\em parameter} $\mu$ are defined by
\[
\mathrm{sn}(u,\mu) = \sin \varphi, \quad
\mathrm{cn}(u,\mu) = \cos \varphi, \quad
\mathrm{dn}(u,\mu) = (1-\mu \sin^2 \varphi)^{\frac{1}{2}}.
 \]

The integral $F(\varphi, \mu)$ given by the formula above is called the {\em incomplete
elliptic integral of the first kind}. The {\em complete} elliptic integral of the first kind
is defined by $K(\mu) = F(\frac{\pi}{2}, \mu)$. The integral defined by
\[
  E(\varphi, \mu) = \int_0^\varphi (1-\mu \sin^2 \theta)^{\frac{1}{2}} d\theta
   \]
is called the {\em incomplete elliptic integral of the second kind}.
The {\em complete} elliptic integral of the second kind is defined by
$E(\mu) =E(\frac{\pi}{2},\mu)$. The integral defined by
\[
  \Pi(n,\varphi, \mu) = \int_0^\varphi \frac{d\theta}{(1-n \sin^2 \theta)(1-\mu \sin^2 \theta)^{\frac{1}{2}}}
   \]
is called the {\em incomplete elliptic integral of the third kind}.
The {\em complete} elliptic integral of the third kind is defined by
$\Pi(n,\mu) = \Pi(n,\frac{\pi}{2},\mu)$.
\end{remark}

%
%

In the following, we will denote by $S_{\mu,r}$
the pseudospherical helicoid associated with the potential $\phi_{\mu,r}$ and
parametrized by the map $f_{\mu,r}$. We will briefly refer to $S_{\mu,r}$ as the
{\em pseudospherical helicoid with parameters} $(\mu,r)$.

%
%

\begin{remark}
If $r=0$, the pseudospherical helicoid reduces to a pseudospherical surface of revolution, while
if $\mu=1$ we obtain the well-known pseudospherical helicoids of Dini \cite{Bi,Fischer,RMcL}.
If  $\mu=1$ and $r=0$, we obtain the Beltrami pseudosphere
\cite{Bi,Ei}. In the following, we will assume that $\mu\neq 1$ and $r\neq 0$.
\end{remark}

Let $S_{\mu,r}$ be
a pseudospherical helicoid with parameters
$(\mu,r)$ and let $\Phi_{\mu,r}$ be the linear change of variables $(s,t) =  \Phi_{\mu,r}(u,v)$, where
\[
 \Phi_{\mu,r}(u,v)=
\left(-\sqrt{|\mu|}\sinh(r)u+\sqrt{|\mu|}\cosh(r)v,\sqrt{|\mu|}\cosh(r)u-\sqrt{|\mu|}\sinh(r)v\right).
  \]
It follows that the fundamental forms of the
reparametrization
 $\verb"L"_{\mu,r}=f_{\mu,r}\circ \Phi_{\mu,r}$ are
\begin{equation}\label{FFL}
\begin{split}
\mathbf{g}_{\mu,r}&=|\mu|\big(\mathrm{sn}^2(u,\mu)+\sinh^2(r)\big)du^2-\frac{|\mu|\sinh(2r)}{2}dudv\\
&\qquad+|\mu|\big(\mathrm{cn}^2(u,\mu)+\sinh^2(r)\big)dv^2,\\
\mathbf{h}_{\mu,r}&=\mu\,\mathrm{sn}(u,\mu)\mathrm{cn}(u,\mu)(du^2-dv^2).
\end{split}
\end{equation}

We can now prove the following.

\begin{prop}\label{R1}
$\mathrm{(1)}$ A pseudospherical helicoid $S_{\mu,r}$ with parameters $(\mu,r)$, $\mu<0$,
and a pseudospherical helicoid $S_{\mu',r}$ with parameters $(\mu',r)$, $\mu'=\mu/(\mu-1)\in (0,1)$,
are mirror images of each other.
$\mathrm{(2)}$ A pseudospherical helicoid $S_{\mu,r}$ with parameters $(\mu,r)$ and
a pseudospherical helicoid $S_{\mu,-r}$ with parameters $(\mu,-r)$ are mirror images of each other.
\end{prop}

\begin{proof}
$\mathrm{(1)}$ Let $K(\mu)$ be the complete elliptic integral of the first kind with parameter $\mu$.
Using the basic modular transformations of the Jacobian elliptic functions (cf. \cite[p.~38 and p.~77]{La}),
we obtain
\begin{equation}\label{MT1}\begin{split}
\mathrm{cn}(u,\mu)&=\mathrm{sn}(\sqrt{1-\mu}u+K(\mu'),\mu'),\\
\mathrm{sn}(u,\mu)&=-\mathrm{cn}(\sqrt{1-\mu}u+K(\mu'),\mu'),
\end{split}
\end{equation}
and
\begin{equation}\label{MT2}
   K(\mu')=\sqrt{1-\mu}\,K(\mu).
     \end{equation}
Consider the linear change of variables
\[
  \Phi':(u,v)\mapsto \frac{1}{\sqrt{1-\mu}}\Big(K(\mu')-v,-u\Big).
    \]
Then, using (\ref{MT1}) and (\ref{MT2}), we deduce that the fundamental forms of $-\verb"L"_{\mu,r}\circ \Phi'$
coincide with those of $\verb"L"_{\mu',r}$.

$\mathrm{(2)}$ We conclude the proof observing that the fundamental forms of $- \verb"L"_{\mu,r}(u,-v)$
coincide with those of $\verb"L"_{\mu,-r}$.
\end{proof}

\begin{remark}
Henceforth, without loss of generality,
we shall restrict to pseudospherical helicoids with parameters $\mu>0$, $\mu\neq 1$, and $r>0$.
\end{remark}

%
%

\section{Explicit integration
of pseudospherical helicoids}\label{s:ex-par}

In this section we will find explicit parametrizations of pseudospherical helicoids in terms of
elliptic functions and elliptic integrals. To this end,
following Popov \cite[Chapter 3]{Po}, pseudospherical helicoids are divided in two classes.

\begin{defn}
A pseudospherical helicoid with parameters $(\mu,r)$
is said of {\em magnetic type} (resp., {\em electric type}) if $\mu >1$ (resp., $\mu \in (0,1)$).
\end{defn}

\subsection{Pseudospherical helicoids of magnetic type}
%
For $\mu>1$ and $r>0$, let
\begin{equation}\label{CPT}
 \begin{split}
m :=\mu&^{-1}\in (0,1),\\
\psi_{\mu,r}(u)&=\sqrt{m+\sinh^2 (r) }\,\coth(r)\,\Pi(-m\, \mathrm{csch}^2 (r),m)\\
&\qquad -\sqrt{m+\sinh^2 (r)}\,\coth(r)\,\Pi(-m\,\mathrm{csch}^2 (r),\mathrm{am}(u,m),m),\\
\xi_{\mu,r}(u)&=\mathrm{dn}(u,m)\,\mathrm{sn}^2(u,m)
\sqrt{m\,\mathrm{sn}^2(u,m)+\sinh^2 (r)},\\
\zeta_{\mu,r}(u)&=\frac{1}{\sqrt{m+\sinh^2 (r)}}\Big(E(m)+\cosh^2(r)(u-K(m))\\
&\qquad -E(\mathrm{am}(u,m),m) \Big),
 \end{split}
 \end{equation}
and
\begin{equation}\label{CPTBis}
\begin{split}
q^1_{\mu,r}(u)&=\frac{-\sqrt{m}}{\sqrt{m+\sinh^2 (r)}\,\mathrm{sn}(u,m)\big(m\,\mathrm{sn}^2(u,m)+\sinh^2(r)\big)},\\
q^2_{\mu,r}(u)&=\frac{\sqrt{m}\cosh(r)\sinh(r)\,\mathrm{cn}(u,m)}
{\big(m+\sinh^2 (r)\big)\,\mathrm{sn}^2(u,m)\,\mathrm{dn}(u,m)\big(m\,\mathrm{sn}^2(u,m)+\sinh^2 (r)\big)},\\
\rho_{\mu,r}=&-2\pi\frac{\cosh(r)\sinh(r)}{m+\sinh^2 (r)}.
\end{split}
\end{equation}
Let $\widetilde{\Gamma}_{\mu,r}\subset \R^3$ be the trajectory of the parametrized curve
\[
\widetilde{\gamma}_{\mu,r}=(\widetilde{x}_{\mu,r},\widetilde{y}_{\mu,r},\widetilde{z}_{\mu,r}):\R\to \R^3
 \]
defined by
\begin{equation}\label{PR1S}
\begin{split}
\widetilde{x}_{\mu,r}&=\xi_{\mu,r}\left(\cos(\psi_{\mu,r})q^1_{\mu,r} - \sin(\psi_{\mu,r})q^2_{\mu,r} \right),\\
\widetilde{y}_{\mu,r}&=\xi_{\mu,r}\left(\sin(\psi_{\mu,r})q^1_{\mu,r} + \cos(\psi_{\mu,r})q^2_{\mu,r} \right),\\
\widetilde{z}_{\mu,r}&=\zeta_{\mu,r}.
\end{split}
\end{equation}

We can now state the following.

\begin{thm}\label{R2}
Let $S_{\mu,r}$ be a pseudospherical helicoid of magnetic type,
that is, a pseudospherical helicoid with parameters $\mu  > 1$ and $r>0$.
Then, the real analytic map $\widetilde{f}_{\mu,r}=
(\widetilde{f}_{\mu,r}^1,\widetilde{f}_{\mu,r}^2,\widetilde{f}_{\mu,r}^3):\R^2\to \R^3$ with components
\[\begin{split}
  \widetilde{f}_{\mu,r}^1(u,v)&=
   \cos(2\pi v)\widetilde{x}_{\mu,r}(u)-\sin(2\pi v)\widetilde{y}_{\mu,r}(u),\\
  \widetilde{f}_{\mu,r}^2(u,v)&= \sin(2\pi v)\widetilde{x}_{\mu,r}(u)+\cos(2\pi v)\widetilde{y}_{\mu,r}(u),\\
  \widetilde{f}_{\mu,r}^3(u,v)&=\widetilde{z}_{\mu,r}(u)+\rho_{\mu,r}v
  \end{split}
     \]
is a global parametrization of $-S_{\mu,r}$.\footnote{Here $-S_{\mu,r}$ stands for $-\text{id}_{\R^3}(S_{\mu,r})$,
where $\text{id}_{\R^3}$ denotes the identity map of $\R^3$.}
\end{thm}


\begin{proof}
Using any software of symbolic computation implementing elliptic functions such as \textsl{Mathematica},
one can see that the coefficients of the fundamental forms of $\widetilde{f}_{\mu,r}$ are
\begin{equation}\label{FNDFR1}
\begin{split}
\widetilde{g}^{\mu,r}_{11}&=m\,\mathrm{sn}^2(u,m)+\sinh^2(r),\\
\widetilde{g}^{\mu,r}_{22}&= \frac{4\pi^2\big(1+\cosh(2r)-2m\mathrm{sn}^2(u,m)\big)}{-1+2m+\cosh(2r)},\\
\widetilde{g}^{\mu,r}_{12}&=\frac{-\sqrt 2\pi \sinh(2r)}{\sqrt{-1+2m+\cosh(2r)}},\\
\widetilde{h}^{\mu,r}_{11}&=-\sqrt{m}\,\mathrm{dn}(u,m)\mathrm{sn}(u,m),\\
\widetilde{h}^{\mu,r}_{22}&=\frac{8\pi^2 \sqrt{m}\,\mathrm{dn}(u,m)\mathrm{sn}(u,m)}{-1+2m+\cosh(2r)},\\
\widetilde{h}^{\mu,r}_{12}&=0.
\end{split}
\end{equation}
Using the parameter
transformations of Jacobian functions (see \cite[p. 77]{La})
\[
  \mathrm{sn}(u,h)=\frac{1}{\sqrt{h}}\,\mathrm{sn}(u \sqrt{h}, h^{-1}),\quad \mathrm{cn}(u,h)
  =\mathrm{dn}(u \sqrt{h},h^{-1})
   \]
and taking into account the expressions (\ref{FFL}), the coefficients $g_{ij}^{\mu,r}$ and $h_{ij}^{\mu,r}$ of the
fundamental forms of $\verb"L"_{\mu,r}$ can be written as
\begin{equation}\label{FFL2}\begin{split}
g^{\mu,r}_{11}&=\mathrm{sn}^2(u/\sqrt{m},m)+\frac{1}{m}\sinh^2 (r) ,\\
g^{\mu,r}_{22}&=\frac{1}{m}\big(\mathrm{dn}^2(u/\sqrt{m},m)+\sinh^2 (r) \big) ,\\
g^{\mu,r}_{12}&=-\frac{1}{m}\sinh(r)\cosh(r),\\
h^{\mu,r}_{11}&=- h^{\mu,r}_{22} =\frac{1}{\sqrt{m}}\,\mathrm{dn}(u/\sqrt{m},m)\,\mathrm{sn}(u/\sqrt{m},m) ,\\
h^{\mu,r}_{12}&=0.
\end{split}
\end{equation}
Next, consider the linear change of variables
\[
  \Phi'': \R^2 \ni (u,v)\mapsto \sqrt{m}\left(u,\frac{2\sqrt{2}\pi}{\sqrt{2m-1+\cosh(2r)}}v \right)\in \R^2.
    \]
From (\ref{FFL2}) it follows that the first quadratic form of the reparametrization
$L_{\mu,r}\circ \Phi''$ coincide with the first quadratic form of $\widetilde{f}_{\mu,r}$ and that the second quadratic form of $L_{\mu,r}\circ \Phi''$ is the opposite of the second quadratic form of $\widetilde{f}_{\mu,r}$. This proves the result.
\end{proof}


\subsection{Pseudospherical helicoids of electric type}

For given $\mu\in (0,1)$ and $r>0$, let
\begin{equation}\label{CPBIS}
\begin{split}
m:=\mu &^{-1}>1,\\
\psi_{\mu,r}(u)&=-\frac{\sqrt{m+\sinh^2 (r)}\,\coth(r)}{\sqrt{m}}
\Pi(-\mathrm{csch}^2(r),\mathrm{am}(\sqrt{m}u,\frac{1}{m}),\frac{1}{m}),\\
\zeta_{\mu,r}(u)&=\frac{(\cosh^2(r)+m-1)u-
\sqrt{m}\,E(\mathrm{am}(\sqrt{m}u,\frac{1}{m}),\frac{1}{m})}{\sqrt{m+\sinh^2(r)}},\\
\end{split}
\end{equation}
Note that $\psi_{\mu,r}$ is quasi-periodic, with quasi-period $2K(1/m)/\sqrt{m}$.
Next, define the real-analytic functions
\begin{equation}\label{CPBIS}
\begin{split}
\xi_{\mu,r}(u)&=\sqrt{\frac{m}{m\,\mathrm{sn}^2(u,m)+\sinh^2(r)}},\\
q^1_{\mu,r}(u)&=\frac{\cosh(r)\sinh(r)}{m+\sinh^2(r)}\mathrm{cn}(u,m),\\
q^2_{\mu,r}(u)&= \frac{\mathrm{dn}(u,m)\,\mathrm{sn}(u,m)}{\sqrt{m+\sinh^2(r)}},
\end{split}
\end{equation}
and the constant
$$\\
\rho_{\mu,r}=-2\pi\frac{\cosh(r)\sinh(r)}{m+\sinh^2(r)}.
$$
Let $\widetilde{\Gamma}_{\mu,r}\subset \R^3$ be the trajectory of the parametrized curve
\[
 \widetilde{\gamma}_{\mu,r}=(\widetilde{x}_{\mu,r},\widetilde{y}_{\mu,r},
  \widetilde{z}_{\mu,r}):\R\to \R^3,
  \]
defined by
\begin{equation}\label{PR1S}
 \begin{split}
\widetilde{x}_{\mu,r}&=\xi_{\mu,r}\left(\cos(\psi_{\mu,r})q^1_{\mu,r} - \sin(\psi_{\mu,r})q^2_{\mu,r} \right),\\
\widetilde{y}_{\mu,r}&=\xi_{\mu,r}\left(\sin(\psi_{\mu,r})q^1_{\mu,r} + \cos(\psi_{\mu,r})q^2_{\mu,r} \right),\\
\widetilde{z}_{\mu,r}&=\zeta_{\mu,r}.
 \end{split}
   \end{equation}

We are in a position to prove the following.

\begin{thm}\label{R3}
Let $S_{\mu,r}$ be a pseudospherical helicoid of electric type, that is, a
pseudospherical helicoid with parameters $\mu\in (0,1)$ and $r>0$.
Then, the real-analytic map $\widetilde{f}_{\mu,r}=
(\widetilde{f}_{\mu,r}^1,\widetilde{f}_{\mu,r}^2,\widetilde{f}_{\mu,r}^3):\R^2\to \R^3$ with components
\[\begin{split}
  \widetilde{f}_{\mu,r}^1(u,v)&=
   \cos(2\pi v)\widetilde{x}_{\mu,r}(u)-\sin(2\pi v)\widetilde{y}_{\mu,r}(u),\\
  \widetilde{f}_{\mu,r}^2(u,v)&= \sin(2\pi v)\widetilde{x}_{\mu,r}(u)+\cos(2\pi v)\widetilde{y}_{\mu,r}(u),\\
  \widetilde{f}_{\mu,r}^3(u,v)&=\widetilde{z}_{\mu,r}(u)+\rho_{\mu,r}v
  \end{split}
     \]
is a global parametrization of $-S_{\mu,r}$.
\end{thm}

\begin{proof}
Since the fundamental forms of $\widetilde{f}_{\mu,r}$ are as in (\ref{FNDFR1}), we can argue
as in the proof of Theorem \ref{R2}.
\end{proof}

\begin{remark}
Observe, however, that the profile curves computed in the two theorems above are not planar (see Figure \ref{FIG2}).
\end{remark}

\begin{figure}[ht]
\begin{center}
\includegraphics[height=6cm,width=6cm]{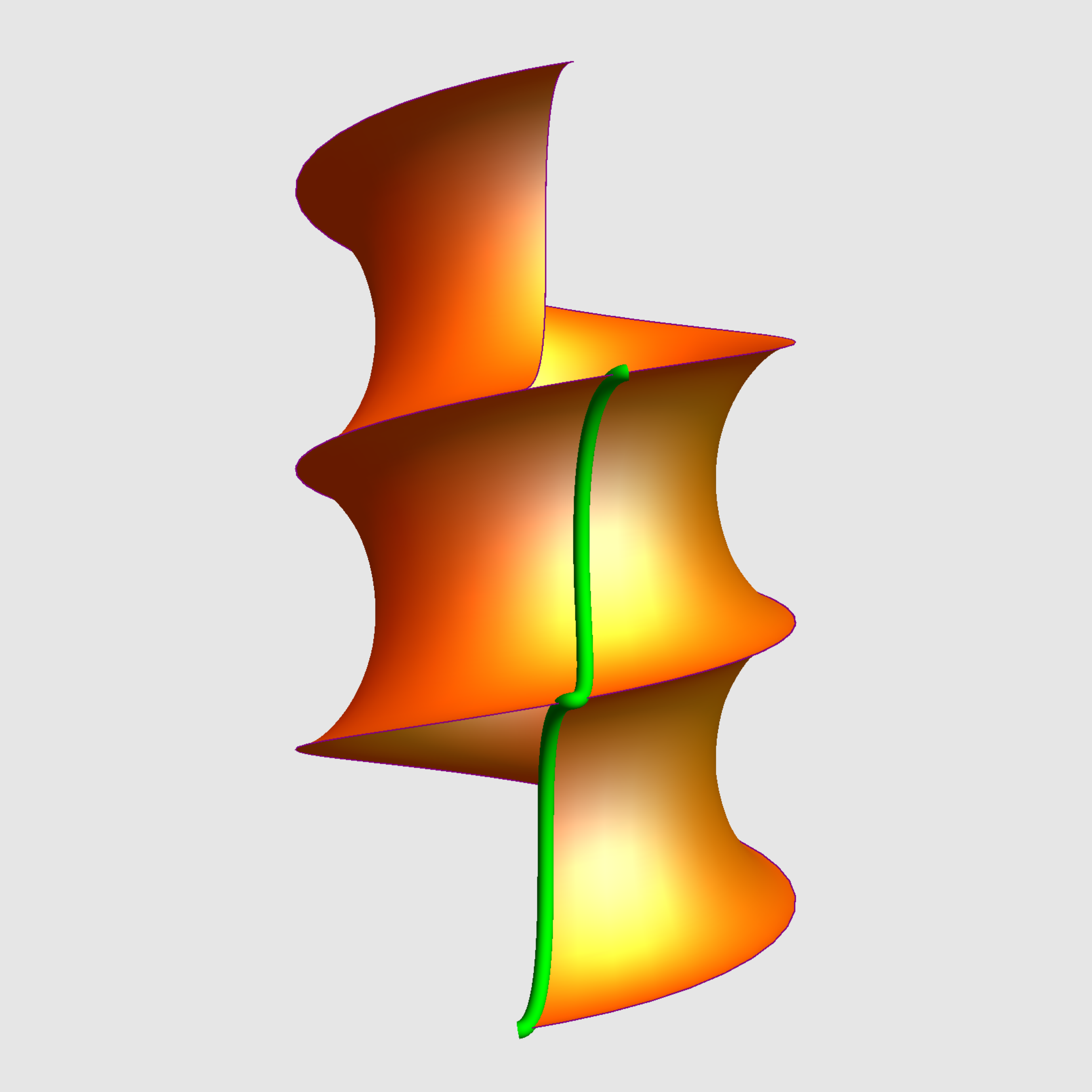}
\includegraphics[height=6cm,width=6cm]{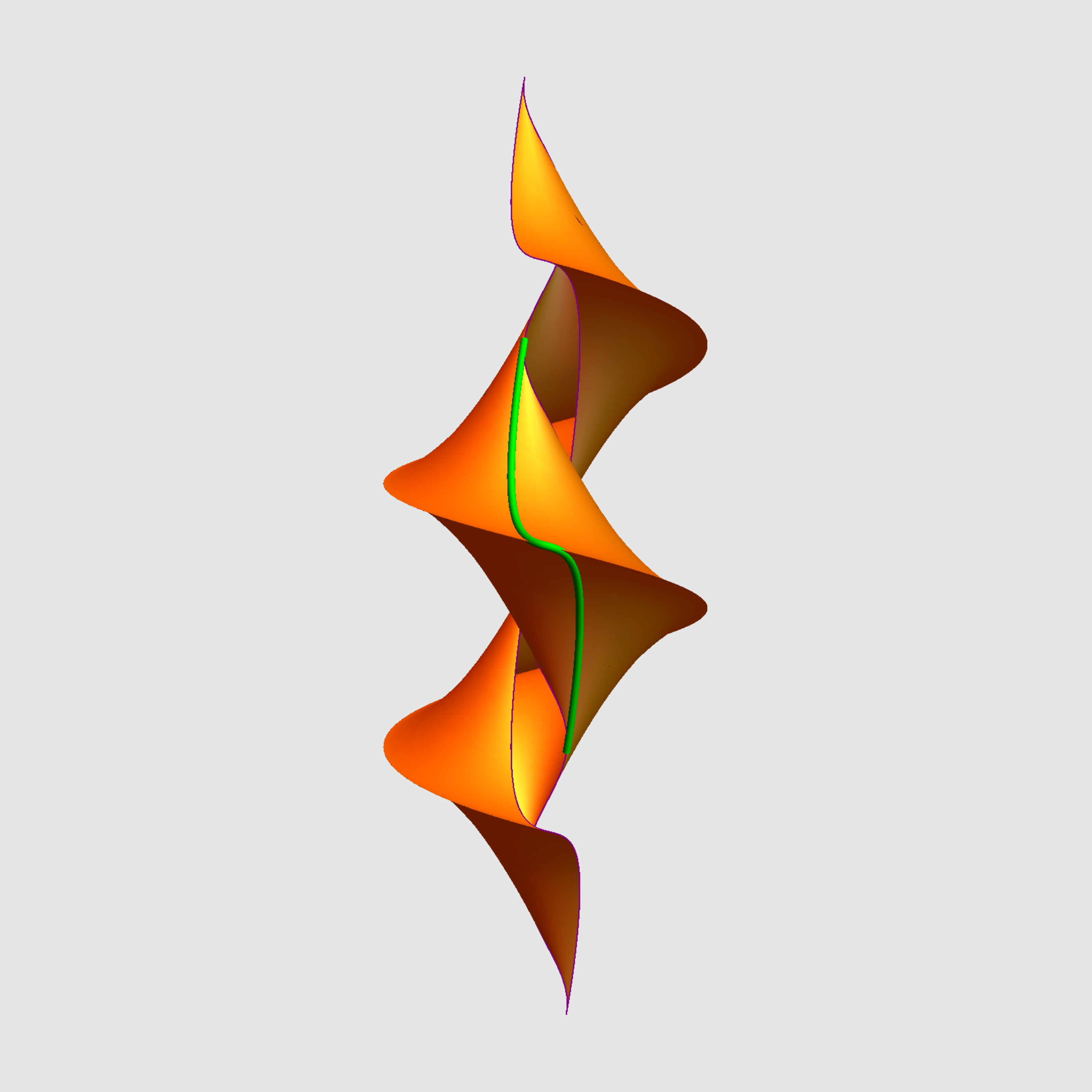}
\caption{The fundamental domains of pseudospherical twisted columns of magnetic (on the left)
and electric (on the right) types and their non-planar profiles.}\label{FIG2}
\end{center}
\end{figure}

\section{The proof of the main theorems}\label{s:proof-mthm}

This section proves Theorems \ref{thm:main-1} and \ref{thm:main}. The proof
is divided into five lemmas.
Lemmas 8 and 9 (cf. Section \ref{ss:plane-profine}) prove that
pseudospherical helicoids of magnetic
and electric type are of translational kind (hence Theorem \ref{thm:main-1})
by determining in both cases
the planar profiles of the pseudospherical helicoids
and by computing
the corresponding phenomenological invariants in terms of the parameters.
Lemmas 10 and 11
(cf. Section \ref{ss:ps-twisted}) prove
the existence of pseudospherical twisted columns of magnetic and electric type with
prescribed wave number and aspect ratio.
Finally, Lemma 12
(cf. Section \ref{ss:conclusion}) proves that
pseudospherical twisted columns with different invariants cannot be congruent.

\subsection{Planar profiles and geometric invariants of pseudospherical helicoids}\label{ss:plane-profine}

We start our analysis with pseudospherical helicoids of magnetic type.

\begin{lemma}
Any pseudospherical helicoid of magnetic type with parameters $\mu >1$ and $r>0$ is of translational kind.
Moreover, it has parity $\epsilon =-1$ and its pitch, wavelength and aspect ratio are given, respectively,  by
\[\begin{split}
\mathfrak{p}_{\mu,r}&=2\pi\frac{\cosh(r)\sinh(r)}{m+\sinh^2(r)},\\
\mathfrak{w}_{\mu,r}&=2\frac{\cosh^2(r)\Pi(-\sinh^2(r),m)-E(m)}{\sqrt{m+\sinh^2(r)}},\\
\mathfrak{d}_{\mu,r}&=\frac{2\sqrt{m+\sinh^2(r)}\big(\cosh^2(r)\Pi(-\sinh^2 r,m)-E(m)\big)}{\sqrt{m(1-m)}},
\end{split}\]
where $m=\mu^{-1}\in (0,1)$
\end{lemma}

\begin{proof}
The profile $\widehat{\Gamma}_{\mu,r}$, although not a plane curve, does not intersect the screw axis.
Therefore, for each $u\in \R$, we may consider the oriented upper half-plane $\mathcal{H}_{\mu,r}(u)$
bounded by the screw axis and passing through the point $\widetilde{\gamma}_{\mu,r}(u)$.
Let $2\pi \theta_{\mu,r}:\R\to \R$ be an analytic determination of the amplitude of the angle between the
oriented upper half-plane
$\{ (x,y,z) \in \R^3 \mid y=0,\, x>0\}$
and the oriented half-plane $\mathcal{H}_{\mu,r}(u)$.
Denote by $[a]$ the integer part of a real number and by $\delta_-:\R\to \R$ the unit step function
\[
 \begin{split}
  \delta_-(u)&=1,\quad u\ge 0,\\
   \delta_-(u)&=0,\quad u< 0.
    \end{split}\]
We can take as $\theta_{\mu,r}$ the unique smooth function
defined on the whole real line, such that,
for each $u\in \mathbb R \setminus\{2hK(m), \, h\in \mathbb Z\}$,
\[
 \theta_{\mu,r}(u)=\frac{1}{2\pi}\left(\psi_{\mu,r}(u)+
 \arctan\left(\frac{q^2_{\mu,r}(u)}{q^1_{\mu,r}(u)}\right)+
  \pi([u/2K(m)]+\delta_-(u))\right).
    \]

\begin{figure}[ht]
\begin{center}
\includegraphics[height=4.5cm,width=8cm]{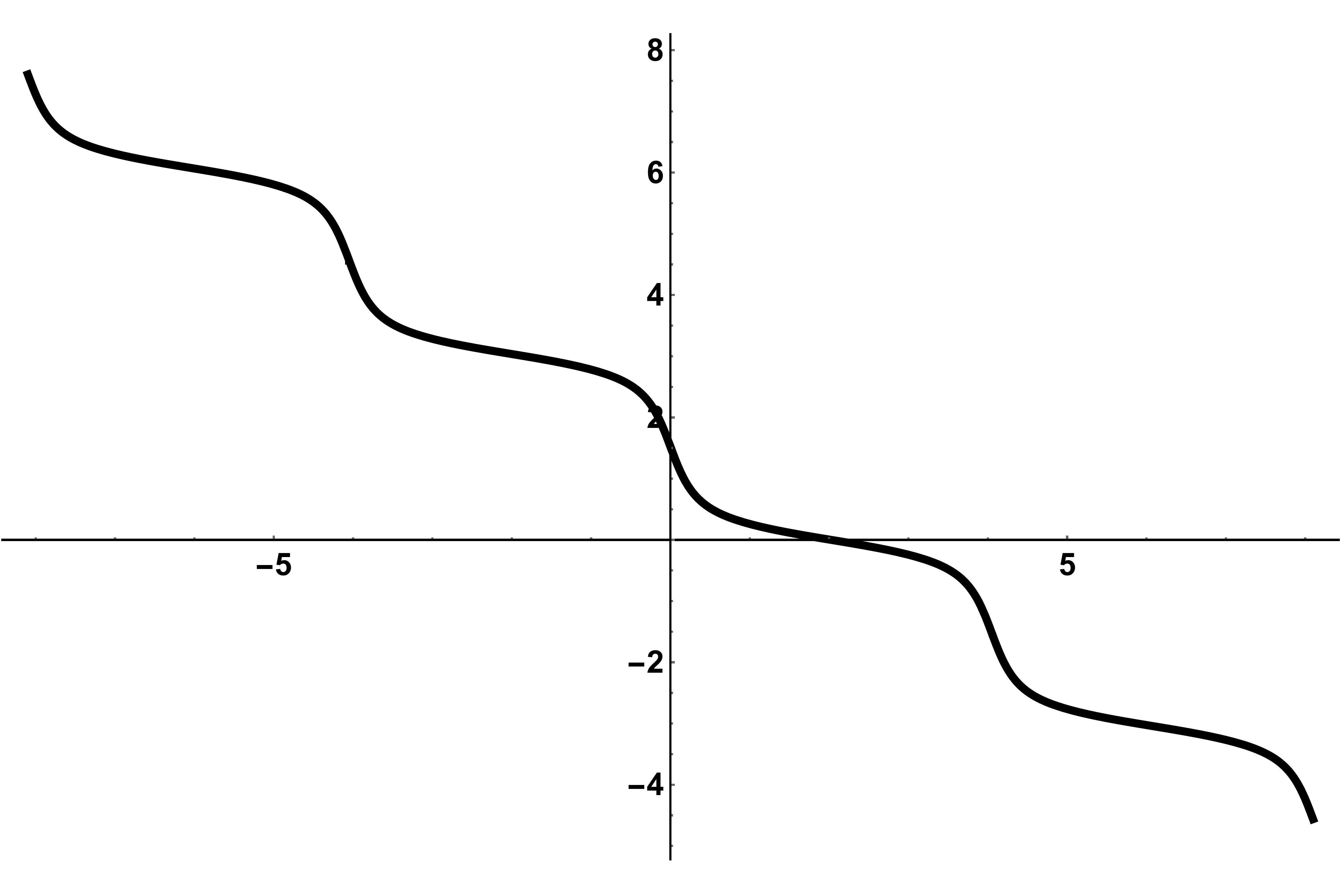}
\caption{The graph of the angular function $\theta_{\mu,r}$, $\mu\approx 1.50742$ and $r\approx 0.208$.}\label{FIG3}
\end{center}
\end{figure}

Consequently, the plane profile of the helicoid is the real-analytic curve $\Gamma_{\mu,r}$ in the upper
half-plane $\{ (x,y,z) \in \R^3 \mid y=0,\, x>0\}$,
parametrized by the map
\[
 \gamma_{\mu,r}=\left(q_{\mu,r},0,
\zeta_{\mu,r} - \rho_{\mu,r}\theta_{\mu,r}\right),
  \]
where
$$
  q_{\mu,r}=\xi_{\mu,r}\sqrt{(q_{\mu,r}^1)^2+(q_{\mu,r}^2)^2}.
    $$
Correspondingly, we obtain the following parametrization of the helicoid in terms of the plane profile
$\Gamma_{\mu,r}$,
{\small
\begin{equation}\label{ppp}
f_{\mu,r}(u,v)=\Big(q_{\mu,r}(u)\cos(2\pi v),
q_{\mu,r}(u)\sin(2\pi v),\zeta_{\mu,r}(u)-\rho_{\mu,r}\theta_{\mu,r}(u) +\rho_{\mu,r} v\Big).
\end{equation}
}
It then follows that the pitch of the helicoid is
$$
   \mathfrak{p}_{\mu,r}=-\rho_{\mu,r}=2\pi\frac{\cosh(r)\sinh(r)}{m+\sinh^2(r)}.
      $$
\begin{figure}[ht]
\begin{center}
\includegraphics[height=6cm,width=3cm]{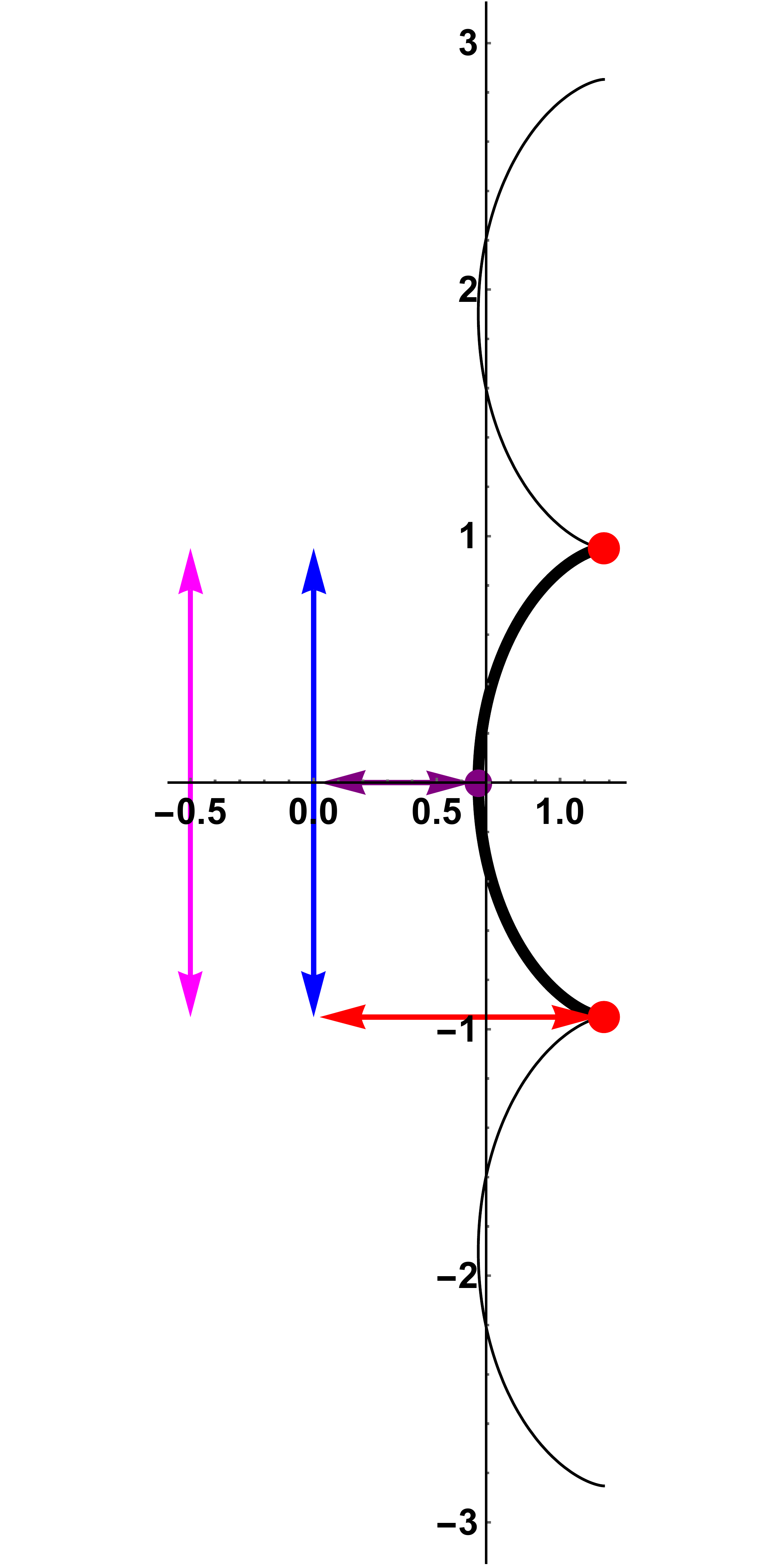}
\includegraphics[height=6cm,width=3cm]{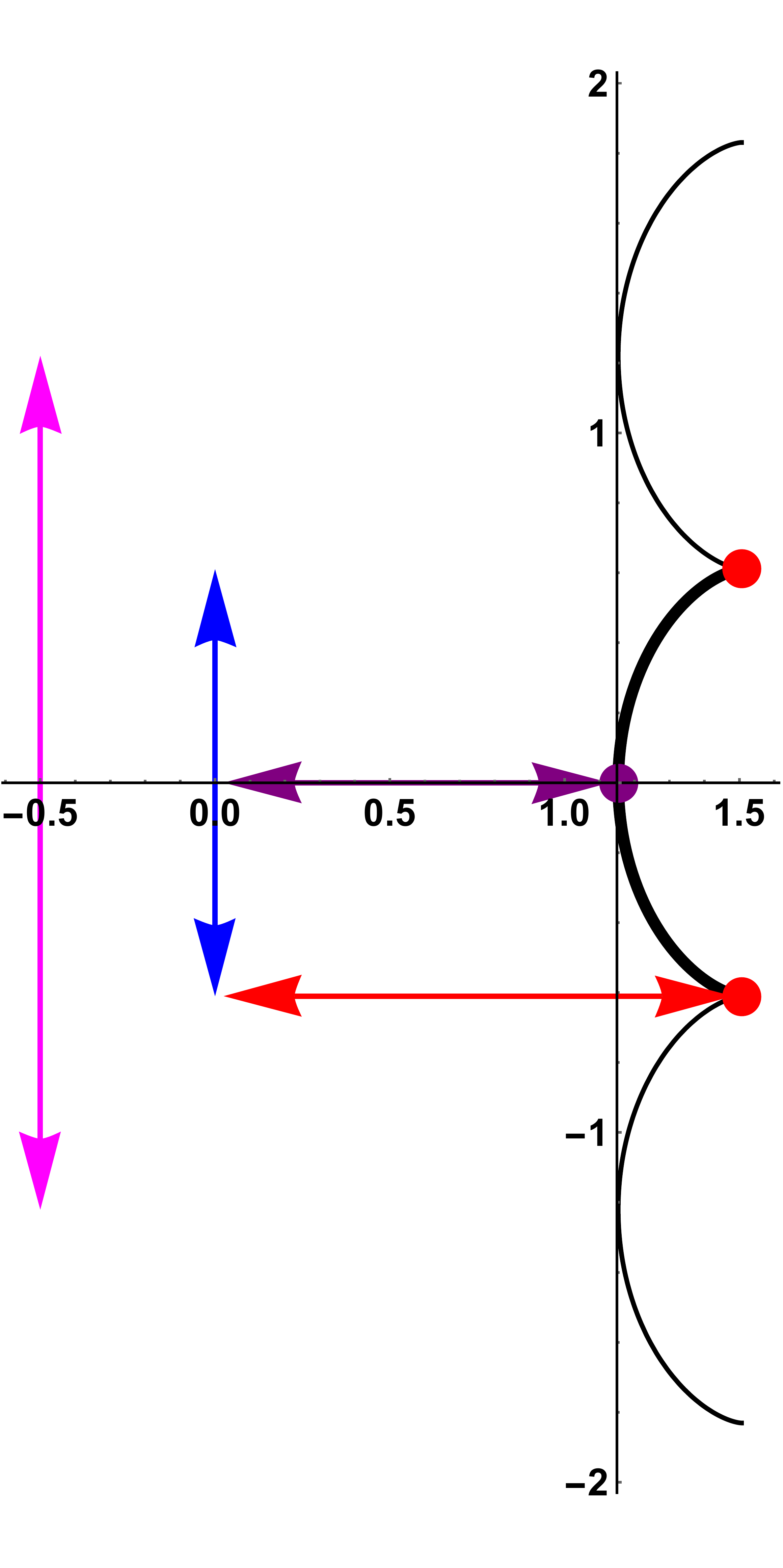}
\caption{The planar profiles of two twisted columns of magnetic type with wave numbers 1 and 2,
respectively.}\label{FIG4}
\end{center}
\end{figure}
The first component of $\gamma_{\mu,r}$ is periodic with minimal period $\omega_{\mu}=2K(m)$
and the third component is a strictly decreasing quasi-periodic function with quasi-period
$\omega_{\mu}$.
Consequently, $\Gamma_{\mu,r}$ is invariant under the translational group along the $z$-axis
and its wavelength is
\[
 \mathfrak{w}_{\mu,r}=d[\gamma_{\mu,r}(2K(m)),\gamma_{\mu,r}(0)]=
2\frac{\cosh^2(r)\Pi(-\sinh^2 (r),m)-E(m)}{\sqrt{m+\sinh^2 (r)}}.
     \]
The profile has countably many cusps of order 1 (ordinary cusps)
located at the points
$$
\gamma_{\mu,r}(2nK(m))=\gamma_{\mu,r}(0)+n\mathfrak{w}_{\mu,r}\vec{k},\quad n\in \Z,
 $$
where $\vec{k}$ denotes the unit $z$-coordinate vector.
All the other points are regular. This implies that the helicoid is of translational type with parity $\epsilon=-1$.
The profile reaches the minimum distance from the screw axis at the points
$$
 \gamma_{\mu}(K(m)+2nK(m))=\gamma_{\mu,r}(K(m))+n\mathfrak{w}_{\mu,r}\vec{k}.
  $$
The inner radius is thus
$$
  \mathfrak{r}^{-}_{\mu,r}=\sqrt{\frac{m(1-m)}{m+\sinh^2(r)}}.
   $$
This implies that the aspect
ratio\footnote{Since $\epsilon=-1$, the aspect ratio is $\mathfrak{w}/\mathfrak{r}^-$.}
is
$$
   \mathfrak{d}_{\mu,r}=\frac{2\sqrt{m+\sinh^2(r)}\big(\cosh^2(r)\Pi(-\sinh^2(r),m)-E(m)\big)}{\sqrt{m(1-m)}},
     $$
which completes the proof.
\end{proof}

As for the pseudospherical helicoids of electric type, we have the following.

\begin{lemma}
A pseudospherical helicoid of electric type with parameters $\mu\in (0,1)$ and $r>0$ is of translational kind.
Moreover, It has parity $\epsilon=1$ and its pitch, wavelength and aspect ratio are given, respectively, by
\[
\begin{split}
\mathfrak{p}_{\mu,r}& = 2\pi\frac{\cosh(r)\sinh(r)}{m+\sinh^2(r)},\\
\mathfrak{w}_{\mu,r}& = 2\frac{(m+\cosh^2(r)-1)K(\frac{1}{m})-mE(\frac{1}{m})-\cosh^2(r)\Pi(-\mathrm{csch}^2(r),\frac{1}{m})}
{\sqrt{m(m+\sinh^2(r))}},\\
\mathfrak{d}_{\mu,r}& = \frac{2\sqrt{m+\sinh^2(r)}}{m\cosh(r)}
\Big((m+\cosh^2 (r) -1)K(\frac{1}{m})-mE(\frac{1}{m}) \\
&\qquad -\cosh^2(r)\,\Pi(-\mathrm{csch}^2(r),\frac{1}{m})\Big),
\end{split}
\]
where $m=\mu^{-1}>1$
\end{lemma}

\begin{proof}
The reasoning is quite similar to the one in the previous lemma.
The profile $\widetilde{\Gamma}_{\mu,r}$ does not intersect the screw axis and the function
$$
 2\pi\theta_{\mu,r}(u)=
 \left(\psi_{\mu,r}(u)+
    \arctan\left(\frac{q^2_{\mu,r}(u)}{q^1_{\mu,r}(u)}\right)\right)
     $$
is an analytic determination of the amplitude of the angle formed by the upper
half-plane $\{ (x,y,z) \in \R^3 \mid y=0,\, x>0\}$
and the upper half-plane bounded by the screw axis and passing through the point
$\widetilde{\gamma}_{\mu,r}(u)$.
Therefore, the plane profile is the real-analytic curve $\Gamma_{\mu,r}$ in the
upper half-plane $\{ (x,y,z) \in \R^3 \mid y=0,\, x>0\}$,
parametrized by
$$
 \gamma_{\mu,r}=\left(q_{\mu,r},0,
  \zeta_{\mu,r}-\rho_{\mu,r}\theta_{\mu,r}\right),
   $$
where
$$
 q_{\mu,r}=\xi_{\mu,r}\sqrt{(q_{\mu,r}^1)^2+(q_{\mu,r}^2)^2}.
   $$
Therefore, the parametrization of the helicoid in terms of the plane profile is
{\small
\begin{equation}\label{ppp1}
f_{\mu,r}(u,v)=\Big(q_{\mu,r}(u)\cos(2\pi v),
q_{\mu,r}(u)\sin(2\pi v),\zeta_{\mu,r}(u)-\rho_{\mu,r}\theta_{\mu,r}(u)+\rho_{\mu,r}v  \Big).
\end{equation}
}
From this it follows that the pitch of the helicoid is
$$
  \mathfrak{p}_{\mu,r}=-\rho_{\mu,r}=2\pi\frac{\cosh(r)\sinh(r)}{m+\sinh^2(r)}.
   $$

\begin{figure}[ht]
\begin{center}
\includegraphics[height=4.5cm,width=8cm]{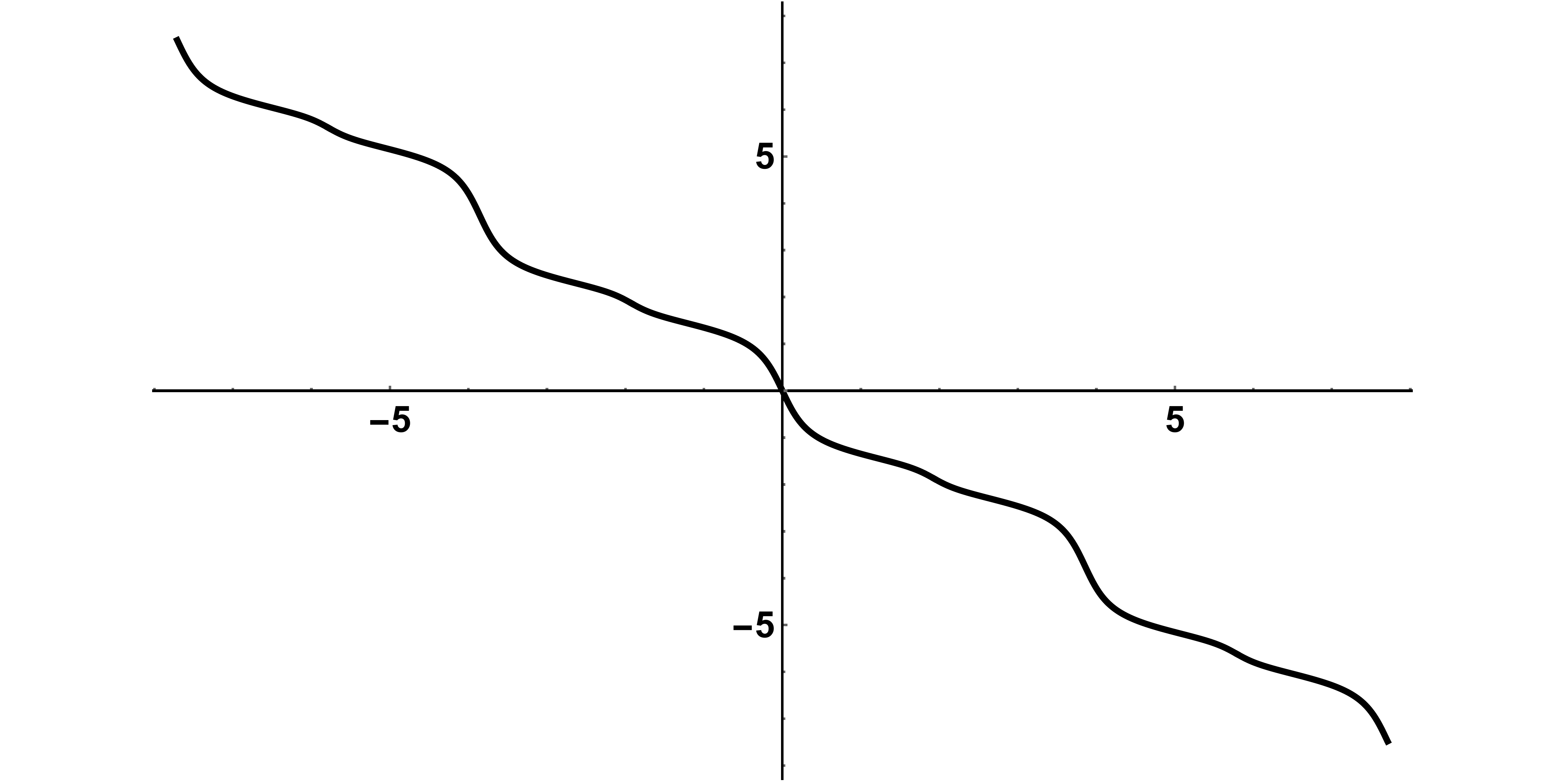}
\caption{The graph of the angular function $\theta_{\mu,r}$,
$\mu\approx 0.7707$ and $r\approx 0.20892$.}\label{FIG5}
\end{center}
\end{figure}

The first component of $\gamma_{\mu,r}$ is periodic with minimal period
$\omega_{\mu}=2K(1/m)/\sqrt{m}$
while the third component is a strictly decreasing quasi-periodic function with
quasi-period $\omega_{\mu}$. Consequently, $\Gamma_{\mu,r}$ is invariant by
the translational group along the $z$-axis and its
wavelength is
\[
 \begin{split}
 \mathfrak{w}_{\mu,r}& =d[\gamma_{\mu,r}(\omega_{\mu}),\gamma_{\mu,r}(0)]\\
& = 2\frac{(m+\cosh^2(r)-1)K( \frac{1}{m})-mE( \frac{1}{m})-\cosh^2(r)\Pi(-\mathrm{csch}^2(r), \frac{1}{m})}
{\sqrt{m(m+\sinh^2(r))}}.
  \end{split}
  \]
The profile possesses countably many cusps of order 1 (ordinary cusps)
located at the points
\[
  \gamma_{\mu,r}(n\omega_{\mu})=\gamma_{\mu,r}(0)+n\mathfrak{w}_{m,r}\vec{k},\quad n\in \Z
   \]
and
\[
\gamma_{\mu,r}\big(\frac{\omega_{\mu}}{2}
  + n\omega_{mu}\big)=\gamma_{\mu,r}\big(\frac{\omega_{\mu}}{2}\big)+n\mathfrak{w}_{m,r}\vec{k},
   \quad n\in \Z
     \]
All the other points are regular. This shows that pseudospherical helicoids of electric
type are of translational kind and have parity $\epsilon=1$. The profile reaches the maximal
distance from the screw axis at the points
\[
\gamma_{\mu}(n\omega_{\mu})=\gamma_{\mu,r}(0)+n\mathfrak{w}_{\mu,r}\vec{k}.
   \]
The outer radius is then
\[
   \mathfrak{r}^{+}_{\mu,r}=\frac{\sqrt{m}\cosh(r)}{m+\sinh^2(r)}.
     \]
This implies that the aspect
ratio\footnote{Since $\epsilon=1$, the aspect ratio is $\mathfrak{w}/\mathfrak{r}^+$.}
is
\[
\begin{split}
   \mathfrak{d}_{\mu,r}=&\frac{2\sqrt{m+\sinh^2(r)}}{m\cosh(r)}
\Big((m+\cosh^2(r)-1)K( \frac{1}{m})-mE( \frac{1}{m})\\
   \qquad &-\cosh^2(r)\Pi(-\mathrm{csch}^2(r), \frac{1}{m})\Big),
   \end{split}
   \]
which proves the claim.
\end{proof}

\begin{figure}[ht]
\begin{center}
\includegraphics[height=6cm,width=3cm]{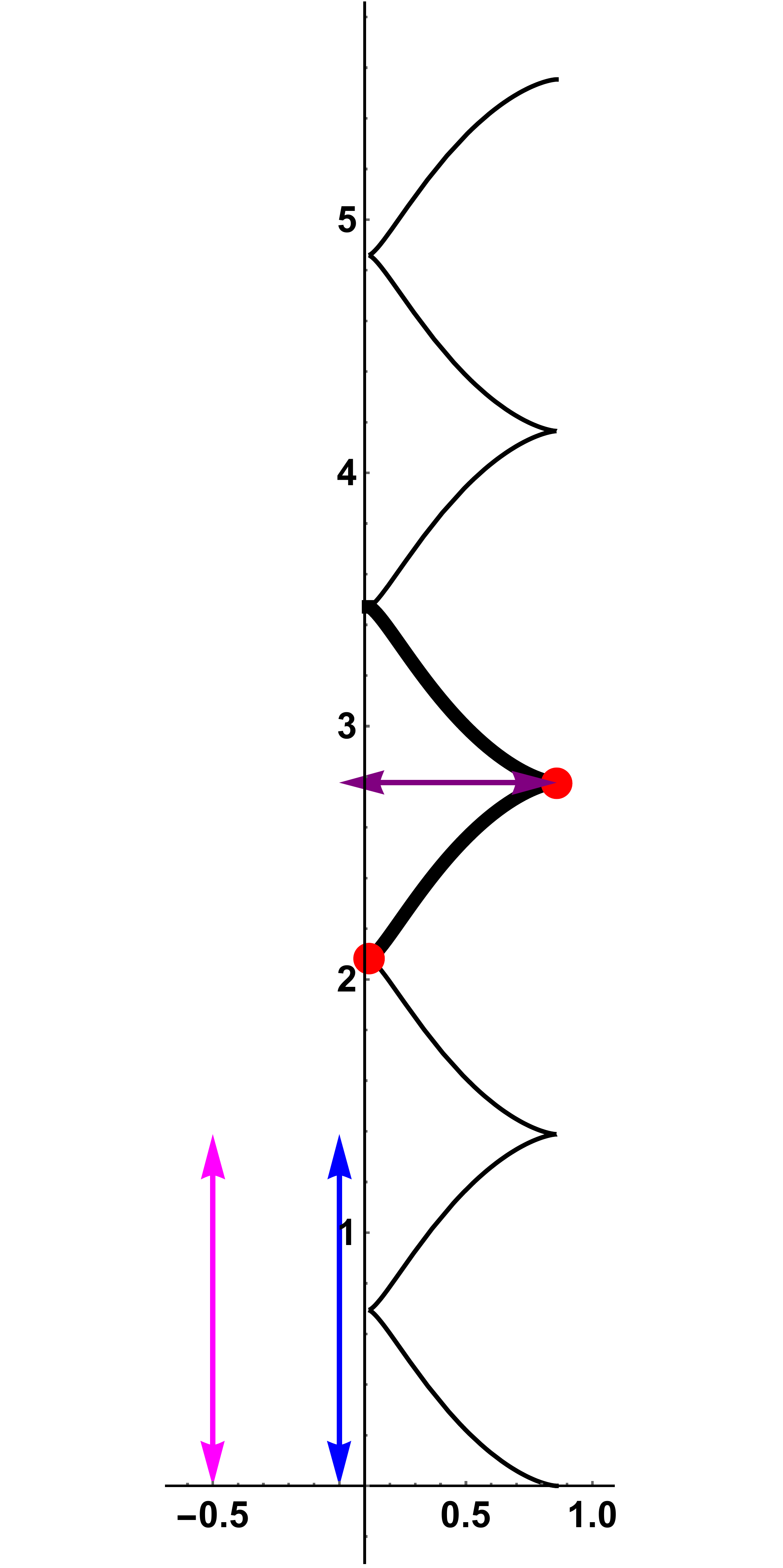}
\includegraphics[height=6cm,width=3cm]{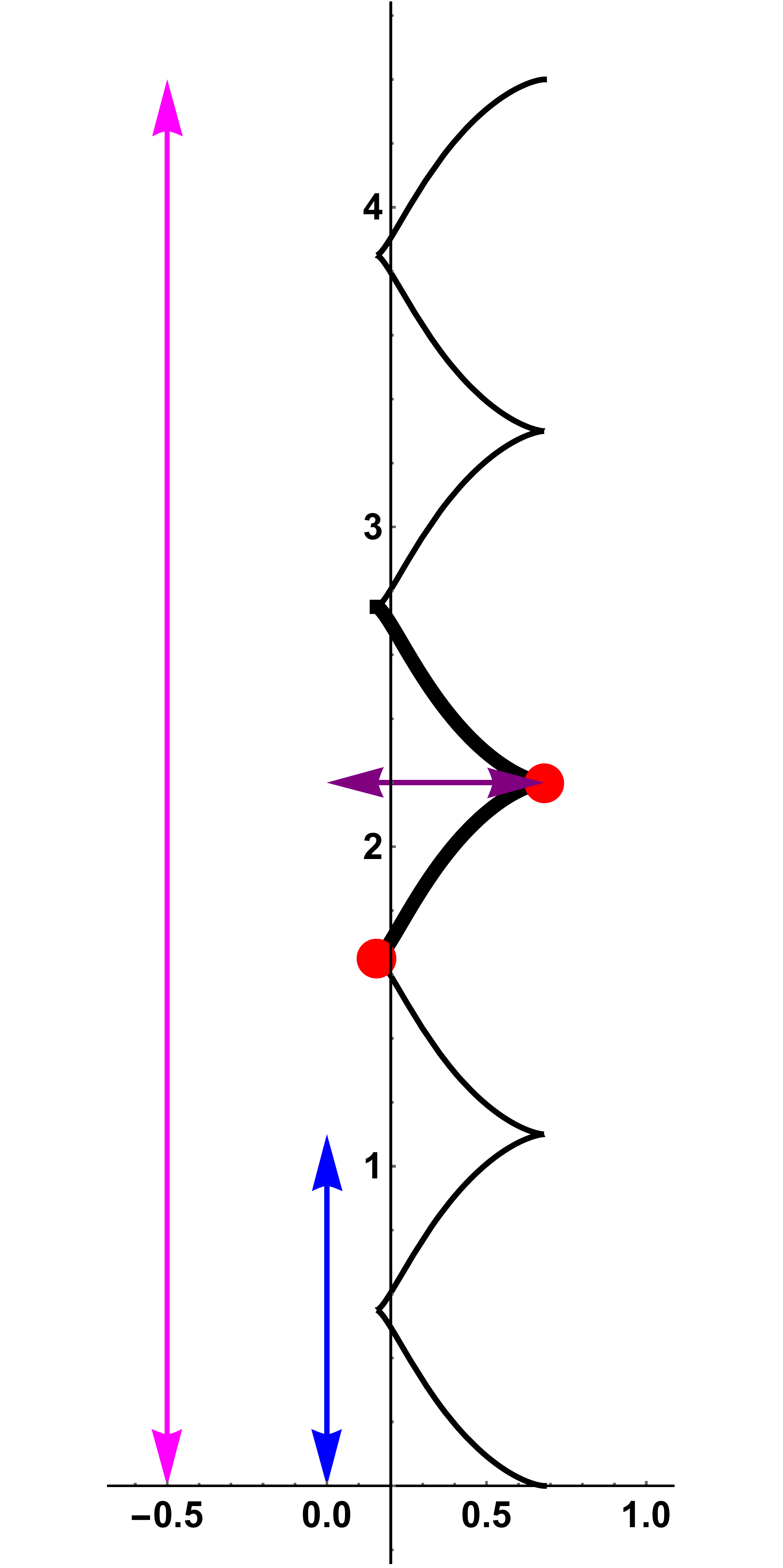}
\caption{The planar profiles of two twisted columns of electric type with wave numbers 1 and 4, respectively.}\label{FIG6}
\end{center}
\end{figure}


\subsection{Existence of pseudospherical twisted columns with prescribed invariants}\label{ss:ps-twisted}

We start by giving the following.

\begin{defn}
Let
\[
  \mathcal{M}_-= \R^+ \times (0,1), \quad   \mathcal{M}_+= \R^+ \times (1,+\infty).
      \]
For a positive integer $n\in \mathbb N$ and a positive real number $d \in \mathbb R^+$,
let
\[
 \begin{split}
  \mathcal{C}^{\pm}_n&=\{(r,m)\in \mathcal{M}_{\pm}\mid \mathfrak{p}_{\mu,r} = n\mathfrak{w}_{\mu,r}\},\\
\mathcal{D}^{\pm}_d&=\{(r,m)\in \mathcal{M}_{\pm}\mid \mathfrak{d}_{\mu,r} = d\},
\end{split}
    \]
where,
depending on whether $\mathcal{C}^{-}_n$, $\mathcal{D}^{-}_d$
or $\mathcal{C}^{+}_n$, $\mathcal{D}^{+}_d$ are considered,
$\mathfrak{p}_{\mu,r}$, $\mathfrak{w}_{\mu,r}$ and $\mathfrak{d}_{\mu,r}$ are as in Lemma 8
or Lemma 9 above. Accordingly, we call
$\mathcal{C}^{-}_n$, $\mathcal{D}^{-}_d$ and $\mathcal{C}^{+}_n$, $\mathcal{D}^{+}_d$
the {\em modular curves of magnetic type} and of {\em electric type}, respectively.
\end{defn}

\begin{remark}
From Lemmas 8 and 9, it follows that a pseudospherical helicoid with parity $\epsilon=-1$
(resp., $\epsilon =1$) is
a twisted column if and only if $(r,m)\in \mathcal{C}^{-}_n$ (resp., $(r,m)\in \mathcal{C}^{+}_n$),
where $n\in \mathbb{N}$
is its wave number.\footnote{Recall that in Lemmas 8 and 9, $m$ stands for $\mu^{-1}$.} Consequently, a
pseudospherical helicoid with parameters $(\mu,r)$ is a twisted column with parity $\epsilon=-1$
(resp., $\epsilon =1$), wave number $\mathfrak{n}$, and aspect ratio $\mathfrak{d}$ if and only
if $(r,m)\in \mathcal{C}^{-}_{\mathfrak{n}}\cap \mathcal{D}^{-}_{\mathfrak{d}}$ (resp.,
$(r,m)\in \mathcal{C}^{+}_{\mathfrak{n}}\cap \mathcal{D}^{+}_{\mathfrak{d}}$).
\end{remark}

\begin{figure}[ht]
\begin{center}
\includegraphics[height=6cm,width=6cm]{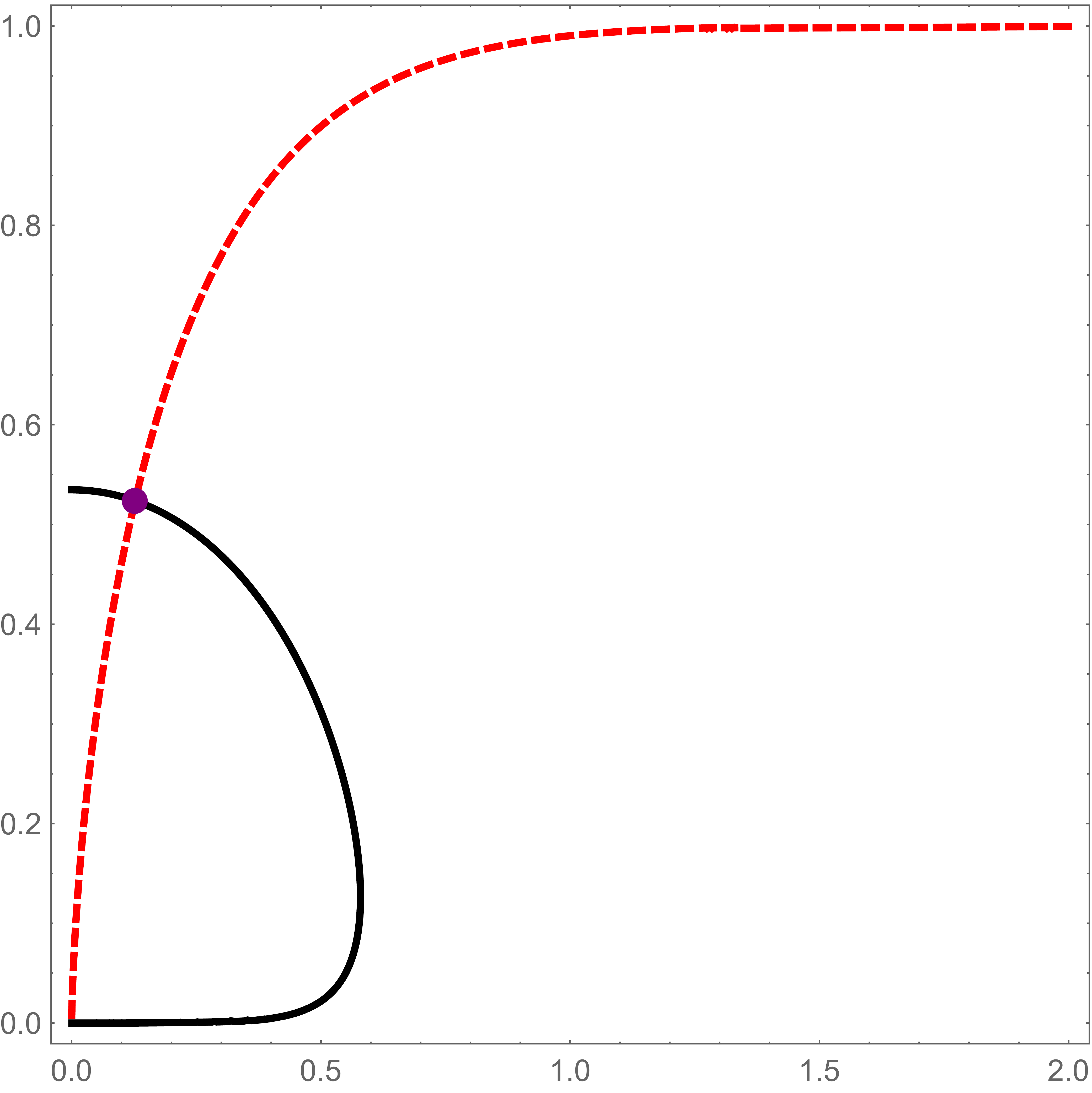}
\includegraphics[height=6cm,width=6cm]{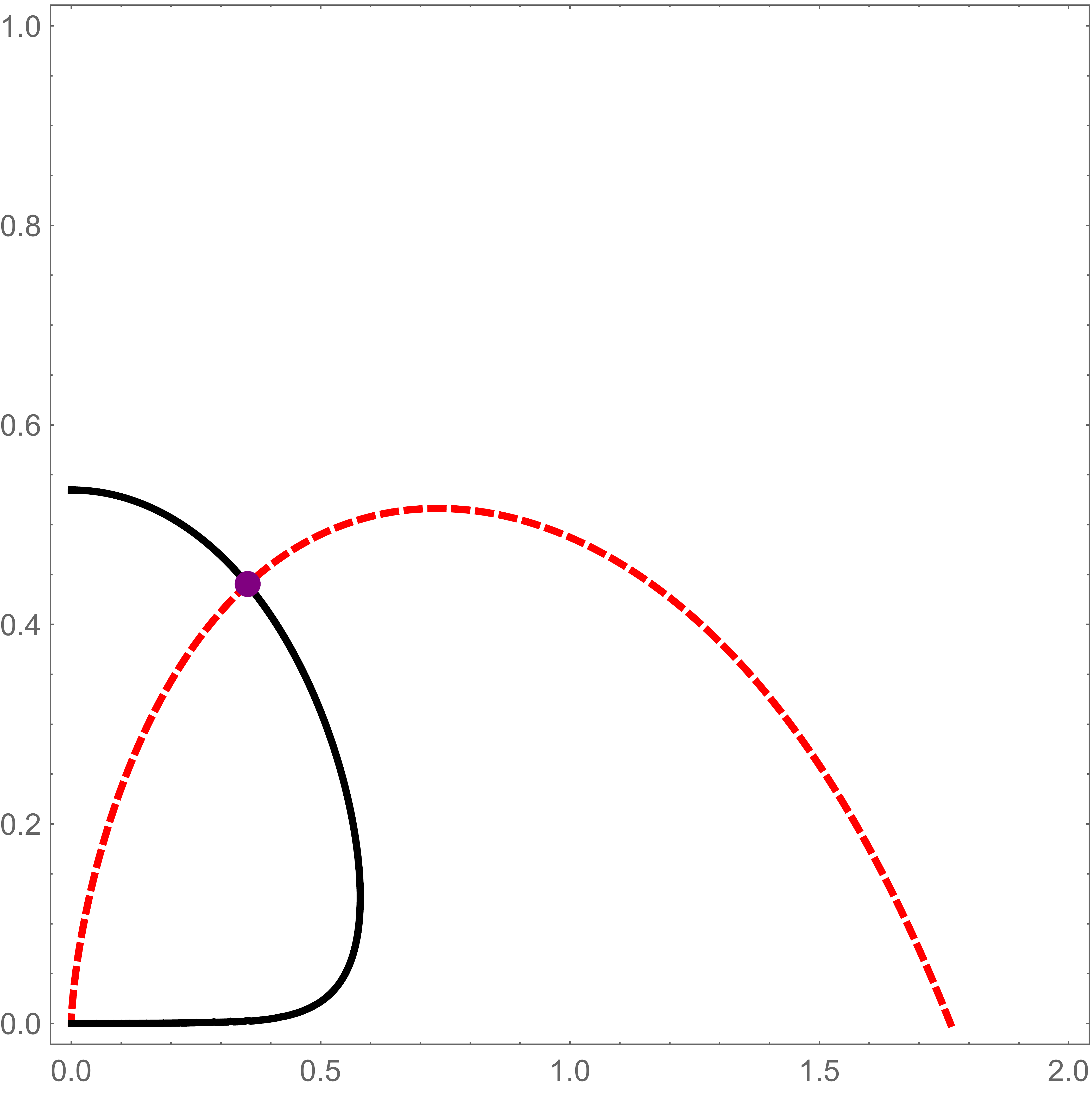}
\caption{The modular curves $\mathcal{C}^-_n$ (dashed lines) and $\mathcal{D}^-_d$ (solid lines) for $n=1$ (left),
$n=3$ (right) and $d$ equal to the golden ratio. The parameters of the corresponding
twisted columns are $\mu \approx 1.90951$, $\mu\approx 2.27181$ and $r=0.127237$, $r=0.353348$,
respectively.}\label{FIG7}
\end{center}
\end{figure}

\begin{lemma}
For given
$n\in \mathbb{N}$ and $d>0$,
there exists a pseudospherical twisted column of magnetic type with wave number
$n$ and aspect ratio $d$.
\end{lemma}

\begin{proof}
It suffices to prove that $\mathcal{C}^-_n\cap \mathcal{D}^-_d$ consists of a single point.
From the previous Lemmas it follows that the points of $\mathcal{C}^-_n\cap \mathcal{D}^-_d$
are the solutions of the system
\begin{equation}\label{system1}\begin{split}
& \cosh^2(r)\Pi(-\sinh^2(r),m)=E(m)+\frac{\pi \cosh(r)\sinh(r)}{n\sqrt{m +\sinh^2(r)}},\\
& \cosh^2(r)\Pi(-\sinh^2(r),m)=E(m)+d\frac{\sqrt{m(1-m)}}{2\sqrt{m+\sinh^2(r)}}.
\end{split}
\end{equation}
This implies that $r=\varrho_{n,d}(m)$, where $\varrho_{n,d}:(0,1)\to \R$ is defined by
\begin{equation}\label{r1}
 \varrho_{n,d}(y)=\frac{1}{2}\mathrm{arcsinh}\left(\frac{d\cdot n\sqrt{y(1-y)}}{\pi} \right).
   \end{equation}
Substituting \eqref{r1} into the second equation of \eqref{system1}, we find that $m\in (0,1)$ is
a zero of the function
\[
\begin{split}
 h_{n,d}(y)&=-2E(y)+2\cosh^2(\varrho_{n,d}(y))\Pi(-\sinh^2(\varrho_{n,d}(y)),y)\\
 &\qquad - d\sqrt{\frac{y(1-y)}{y+\sinh^2(\varrho_{n,d}(y))}},
  \end{split}
    \]
which is defined on $(0,1)$.
The latter is a strictly increasing real-analytic function, such that
\[
 \lim\limits_{y\to 1^-}h_{n,d}(y)=+\infty, \quad
  \lim\limits_{y\to 0^+}h_{n,d}(y)=
  -\frac{2d\pi}{\sqrt{n^2d^2+4\pi^2}}<0.
   \]
Therefore, for given $n\in \mathbb{N}$ and $d>0$, there exists a unique $m_{n,d}\in (0,1)$,
 such that $h_{n,d}(m_{n,d})=0$. We have thus proved that
$$
 \mathcal{C}^-_n\cap \mathcal{D}^-_d=(\varrho_{n,d}(m_{n,d}),m_{n,d}),
 $$
which yields the required result.
\end{proof}

\begin{figure}
\begin{center}
\includegraphics[height=6cm,width=6cm]{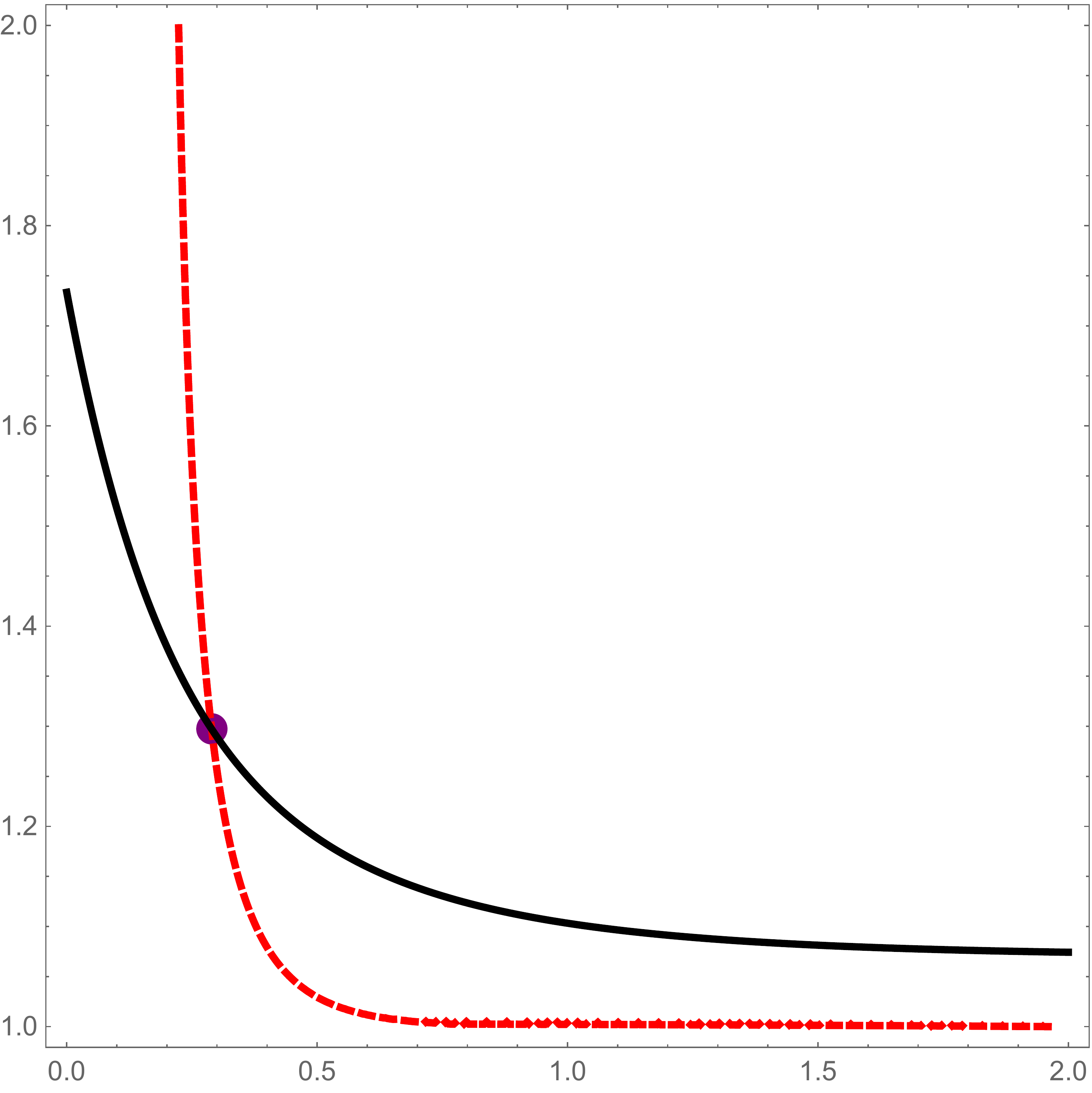}
\includegraphics[height=6cm,width=6cm]{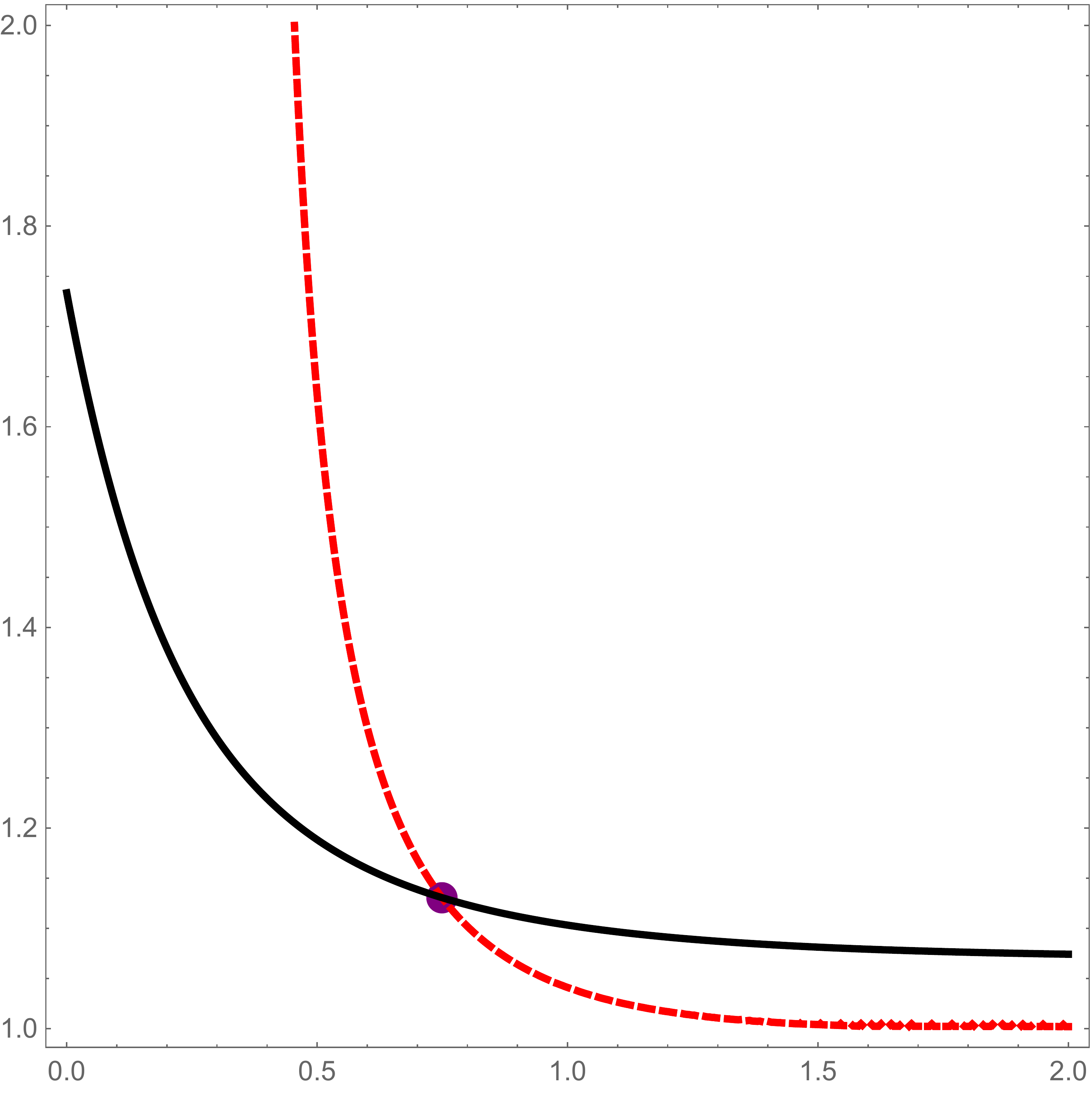}
\caption{The modular curves $\mathcal{C}^+_n$ (dashed lines) and $\mathcal{D}^+_d$ (solid lines)
for $n=1$ (left), $n=3$ (right) and $d$ equal to the golden ratio. The parameters of the corresponding
twisted columns are $\mu \approx 0.7708$, $\mu\approx 0,8844$ and $r=0.289$, $r=0.749$, respectively.}\label{FIG8}
\end{center}
\end{figure}

\begin{lemma}
For given
$n\in \mathbb{N}$ and $d>0$, there exists a pseudospherical twisted column of electric
type with wave number $n$ and aspect ratio $d$.
\end{lemma}

\begin{proof}
We argue as in Lemma 10.
The points of $\mathcal{C}^+_n\cap \mathcal{D}^+_d$ are the solutions of the system
\begin{equation}\label{system11}
\begin{split}
&\frac{\pi\sqrt{m}\sinh(2r)}{\sqrt{m+\sinh^2(r)}}+2n\Big[mE\Big(\frac{1}{m}\Big)-(\cosh^2(r)+m-1)K\Big(\frac{1}{m}\Big)\\
&\qquad +\cosh^2(r) \Pi(-\mathrm{csch}^2(r),\frac{1}{m})\Big]=0,\\
&\frac{dm\cosh(r)}{\sqrt{m+\sinh^2(r)}}+2\Big[mE\Big(\frac{1}{m}\Big)-(\cosh^2(r)+m-1)K\Big(\frac{1}{m}\Big)\\
&\qquad +\cosh^2(r)\Pi(-\mathrm{csch}^2(r),\frac{1}{m})\Big]=0.
\end{split}
\end{equation}
The first equation implies that $\Pi(-\mathrm{csch}^2(r),1/m)$ is equal to
\[
\begin{split}
 \qquad &\frac{n\,\mathrm{sech}^2(r)\big(\cosh(2r)+2m-1\big)\big(-2mE(\frac{1}{m})
 +(\cosh(2r)+2m-1)K(\frac{1}{m})\big)} {4n(m+\sinh^2(r))}\\
 &\qquad - \frac{\pi\sqrt{m(m+\sinh^2(r))}\tanh(r)}{n(m+\sinh^2(r))}.
 \end{split}
  \]
Substituting this expression in the second equation of \eqref{system11}, we deduce that
$r=\varrho_{n,d}(m)$, where $\varrho_{n,d}:\R^+\to \R$ is defined by
\begin{equation}\label{r11}
\varrho_{n,d}(y)=\mathrm{arcsinh}\left(\frac{n\sqrt{y}d}{2\pi} \right).
\end{equation}
Substituting \eqref{r11} into the first equation of \eqref{system11}, we find that $m\in (1,+\infty)$
is a zero of the function
\[
\begin{split}
h_{n,d}(y)&=\frac{2d\pi^2\sqrt{d^2n^2y^2+4\pi^2y}}{\sqrt{d^2n^2+4\pi^2}}
+
\Big(4y\pi^2E\Big(\frac{1}{y}\Big)
-y(d^2n^2+4\pi^2)K\Big(\frac{1}{y}\Big)\\
&\qquad +(d^2n^2y+4\pi^2)\Pi(-\frac{4\pi^2}{d^2n^2y},\frac{1}{y})\Big),
\end{split}
\]
which is defined on $(0,+\infty)$.
This is a strictly increasing real-analytic function such that
$$
 \lim\limits_{y\to 0^+}h_{n,d}(y)=-\infty,\quad
   \lim\limits_{y\to +\infty}h_{n,d}(y)=+\infty.
     $$
Therefore, for given $n\in \mathbb{N}$ and $d>0$, there exists a unique $m_{n,d}\in (1,+\infty)$
such that $h_{n,d}(m_{n,d})=0$. We have thus proved that
\[
 \mathcal{C}^+_n\cap \mathcal{D}^+_d=(\varrho_{n,d}(m_{n,d}),m_{n,d}),
  \]
which yields the required result.
\end{proof}

\begin{figure}
\begin{center}
\includegraphics[height=6cm,width=6cm]{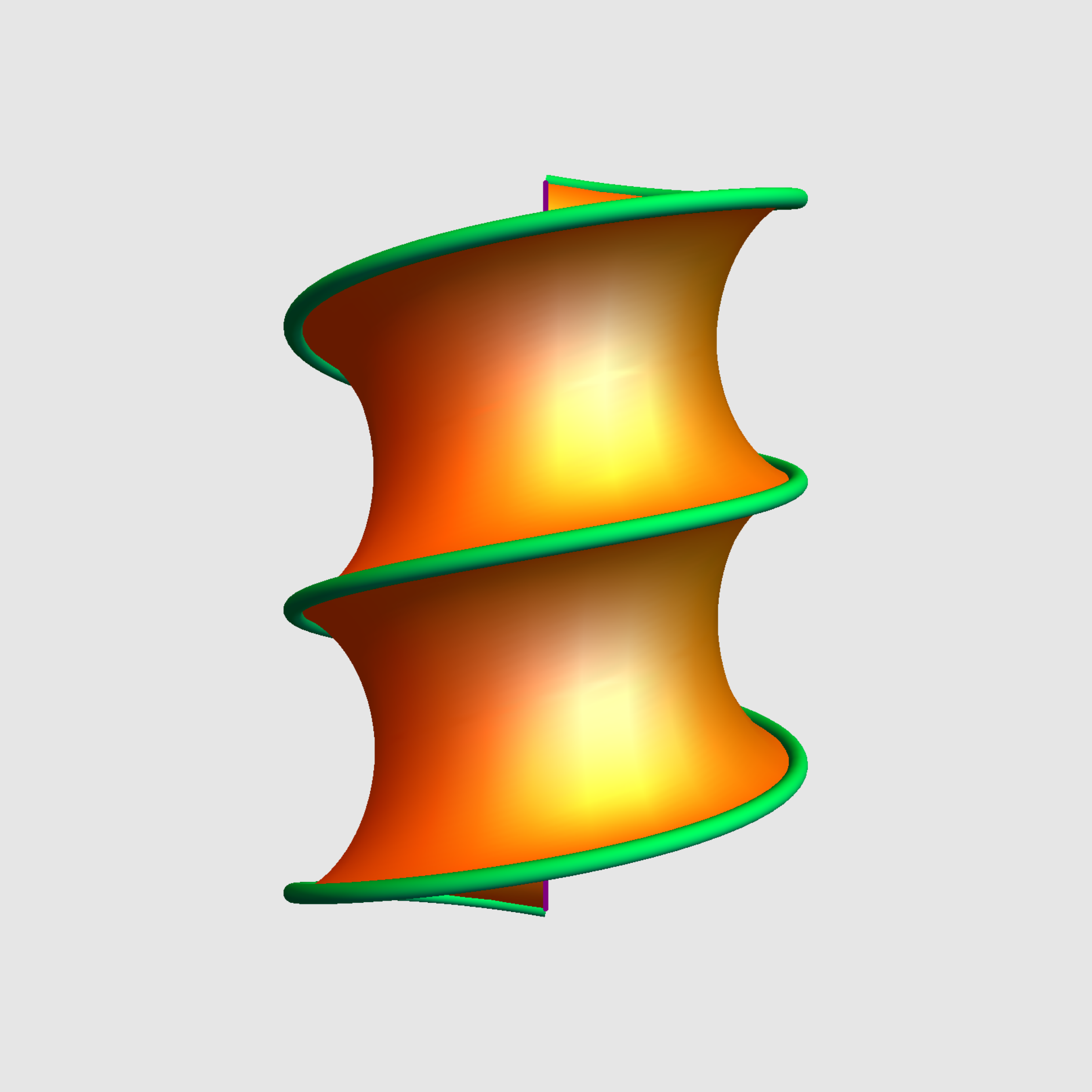}
\includegraphics[height=6cm,width=6cm]{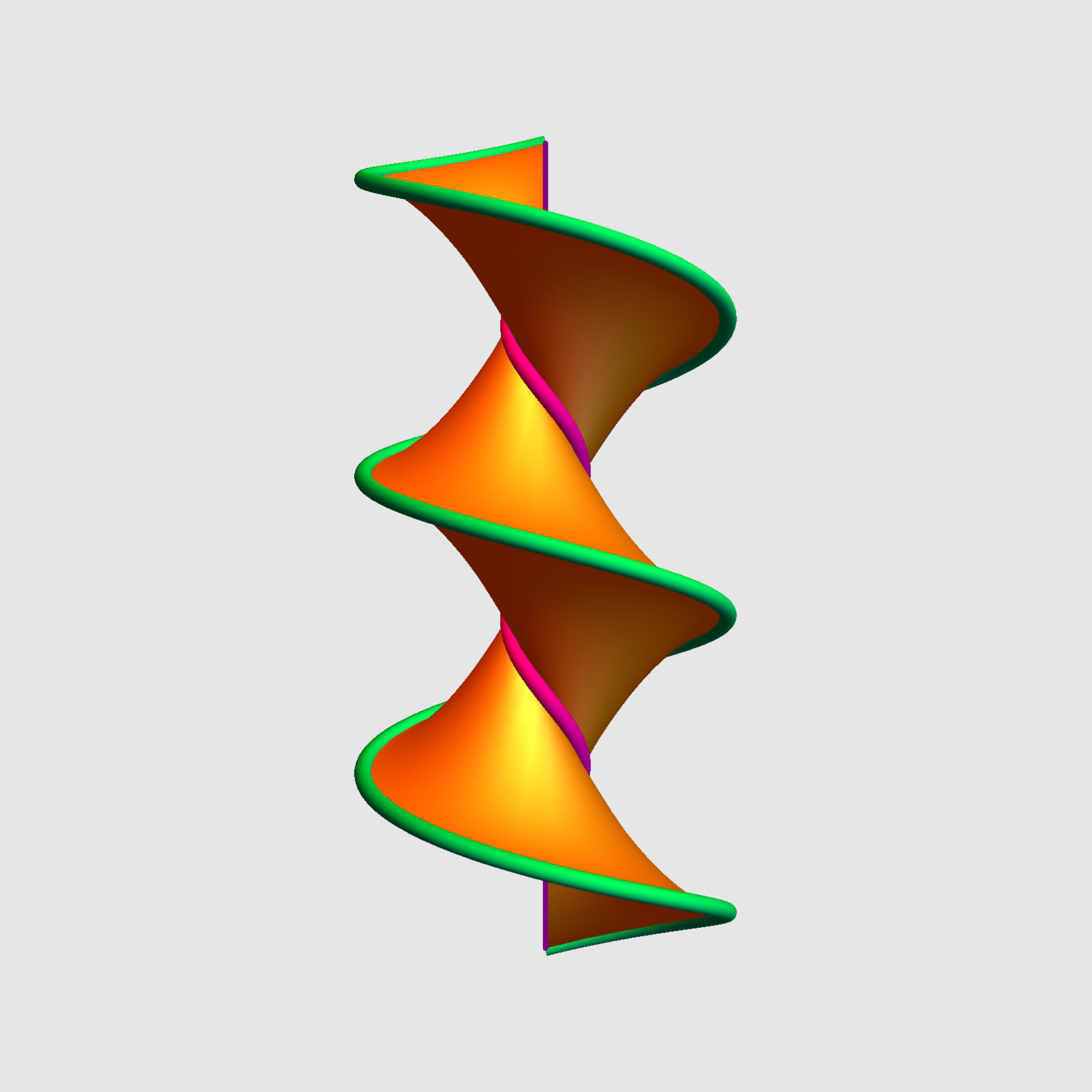}
\caption{Magnetic (left) and electric (right) pseudospherical twisted columns
with $\mathfrak{n}=1$ and $\mathfrak{d}=\phi$, $\phi = \text{golden ratio}$.}\label{FIG9}
\end{center}
\end{figure}

\begin{remark}[Number of singular helices]
In the magnetic-type case, the wave number tells how many singular helices are contained in the
pseudospherical helicoid. In the electric-type case, instead, the number of singular helices
in the pseudospherical helicoid is two times the wave number (see Example \ref{ex1}).
This result should be compared with the local analysis of Brander \cite{Brander} about
the pseudospherical fronts containing circular helices as singular points.
\end{remark}

\subsection{End of the proof}\label{ss:conclusion}

To conclude the proof of Theorem \ref{thm:main} we are left with the following.

\begin{lemma}
Let $S$ and $S'$ be two right-handed (or left-handed) pseudospherical twisted columns with invariants $(\epsilon,\mathfrak{n},\mathfrak{d})$ and $(\epsilon ',\mathfrak{n}',\mathfrak{d}')$. If
$(\epsilon,\mathfrak{n},\mathfrak{d})\neq (\epsilon',\mathfrak{n}',\mathfrak{d}')$ then $S$ and $S'$
 cannot be congruent to each other.
 \end{lemma}

\begin{proof}
First, we show that if $\epsilon\neq \epsilon'$, then $S$ and $S'$ cannot be congruent to each other.
Suppose, on the contrary, that $S$ and $S'$ are congruent. By possibly acting with a rigid motion,
we may assume $S=S'$.
Thus, if $\Gamma_*$ is a fundamental domain of the planar profile of $S$, then $\Gamma_*$ must be
a fundamental domain of the planar profile of $S'$. On the other hand, if $\epsilon\neq \epsilon'$, the number
of cusps of the fundamental domains are different, which is a contradiction.
Next, suppose $\epsilon=\epsilon'$. If $\mathfrak{n}\neq \mathfrak{n}'$ and $\epsilon=\epsilon'$, then
the number of connected components of the singular loci of $S$ and $S'$ are different. Thus $S$ and $S'$
cannot be congruent.
We now prove that if $\epsilon=\epsilon'$ and $\mathfrak{n}=\mathfrak{n}'$, but
$\mathfrak{d}\neq \mathfrak{d}'$, then $S$ and $S'$ are not congruent.
Suppose not, then they must have the same pitch and the same inner and outer radii.
If $\mathfrak{p}=\mathfrak{p}'$ and $\mathfrak{n}=\mathfrak{n}'$, then $\mathfrak{w}=\mathfrak{w}'$.
Since they have the same wave numbers, parities, inner and outer radii, they must have also
the same aspect ratios, which is a contradiction.
\end{proof}

\begin{ex}\label{ex1}
We now examine the pseudospherical twisted columns of the golden-mean series
(that is, $\mathfrak{d}$ is equal to the golden ratio). The numerical evaluation of the parameters $(\mu,r)$
and the visualization have been done with the software \textsl{Mathematica}.
The following is the table of the approximate values of the parameters $(\mu,r)$ of the twisted columns
(magnetic type) with phenomenological invariants $(\epsilon,\mathfrak{n},\mathfrak{d})$
$=$ $(-1,n,\phi)$, $n=1,\dots, 6$,
\[
\begin{tabular}{|c|c|c|c|c|c|c|c}
  \hline
 $n$ & 1 & 2 & 3 & 4 & 5 & 6 \\
  $\mu$ & 1.90951 & 2.03576 & 2.27181 & 2.65802 & 3.25281 & 4.13094 \\
  $r$ & 0.127237 & 0.247275 & 0.353348 & 0.439982 & 0.504218 & 0.546398 \\
  \hline
\end{tabular}
\]
The following is the table of the approximate values of the parameters $(\mu,r)$ of the twisted columns
(electric type) with phenomenological invariants $(\epsilon,\mathfrak{n},\mathfrak{d})$ $=$
$(1,n,\phi)$, $n=1,\dots, 6$,
\[
\begin{tabular}{|c|c|c|c|c|c|c|c}
  \hline
 $n$ & 1 & 2 & 3 & 4 & 5 & 6 \\
  $\mu$ & 0.770862& 0.849115 & 0.884453 & 0.902344 & 0.912336 & 0.918378\\
  $r$ &  0.289255& 0.533287 & 0.749346 & 0.939799 &  1.1074 & 1.25549\\
  \hline
\end{tabular}
\]

Figure \ref{FIG9} reproduces pseudospherical twisted columns with invariants $(-1,1,\phi)$ and $(1,1,\phi)$,
respectively. The magnetic-type one has one singular helix while the electric-type one has two singular helices.
Figure \ref{FIG10} reproduces pseudospherical twisted columns with invariants $(-1,2,\phi)$ and $(1,2,\phi)$,
respectively.
The magnetic-type one has two singular helices while the electric-type one has four singular helices.
Figure \ref{FIG11} depicts pseudospherical twisted columns with invariants $(-1,3,\phi)$ and $(1,3,\phi)$,
respectively.
The magnetic-type one has three singular helices, while the electric-type one has six singular helices.
For the magnetic-type twisted columns we have chosen the right-handed orientation, while for the
electric-type ones we have chosen the left-handed orientation.
\end{ex}

\begin{figure}
\begin{center}
\includegraphics[height=6cm,width=6cm]{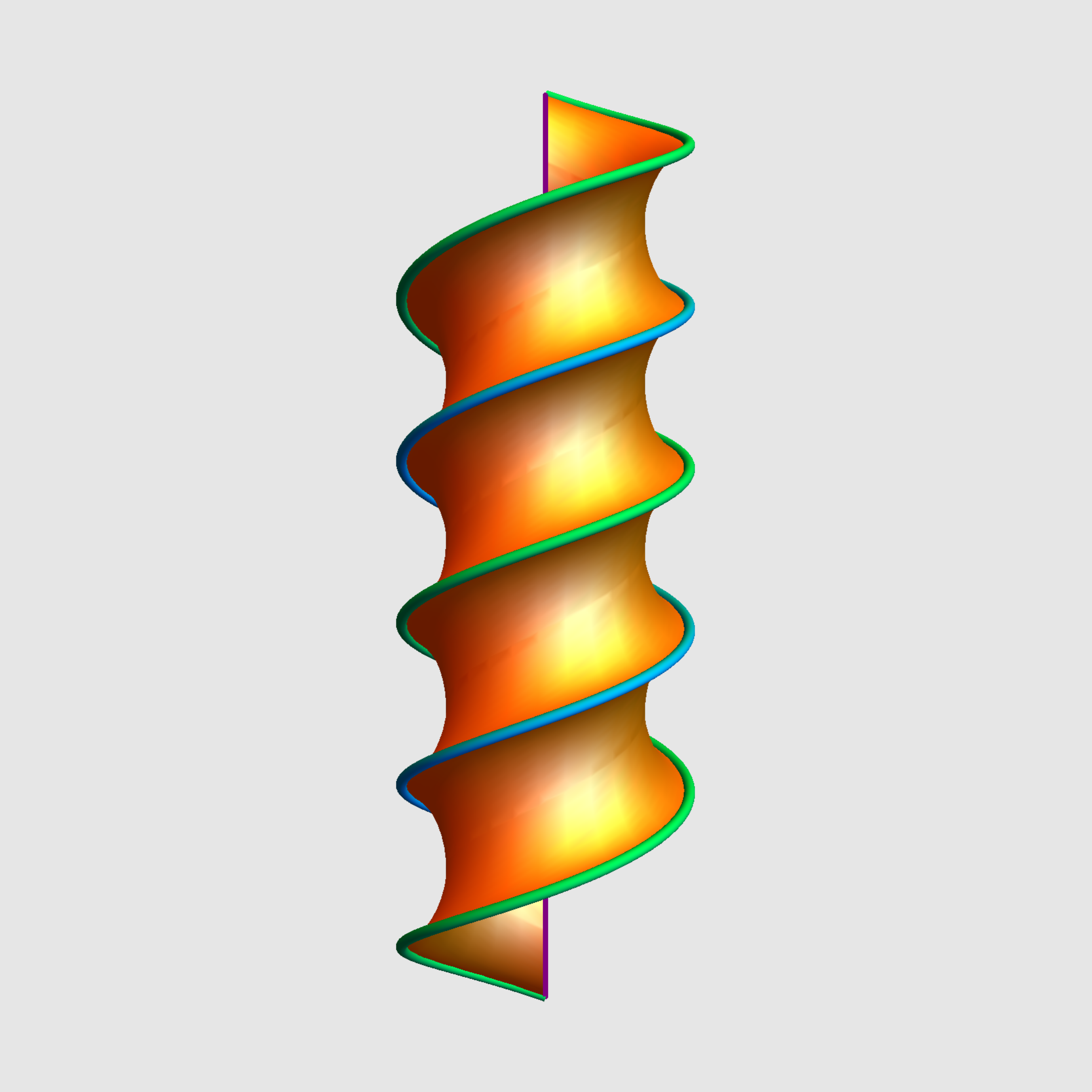}
\includegraphics[height=6cm,width=6cm]{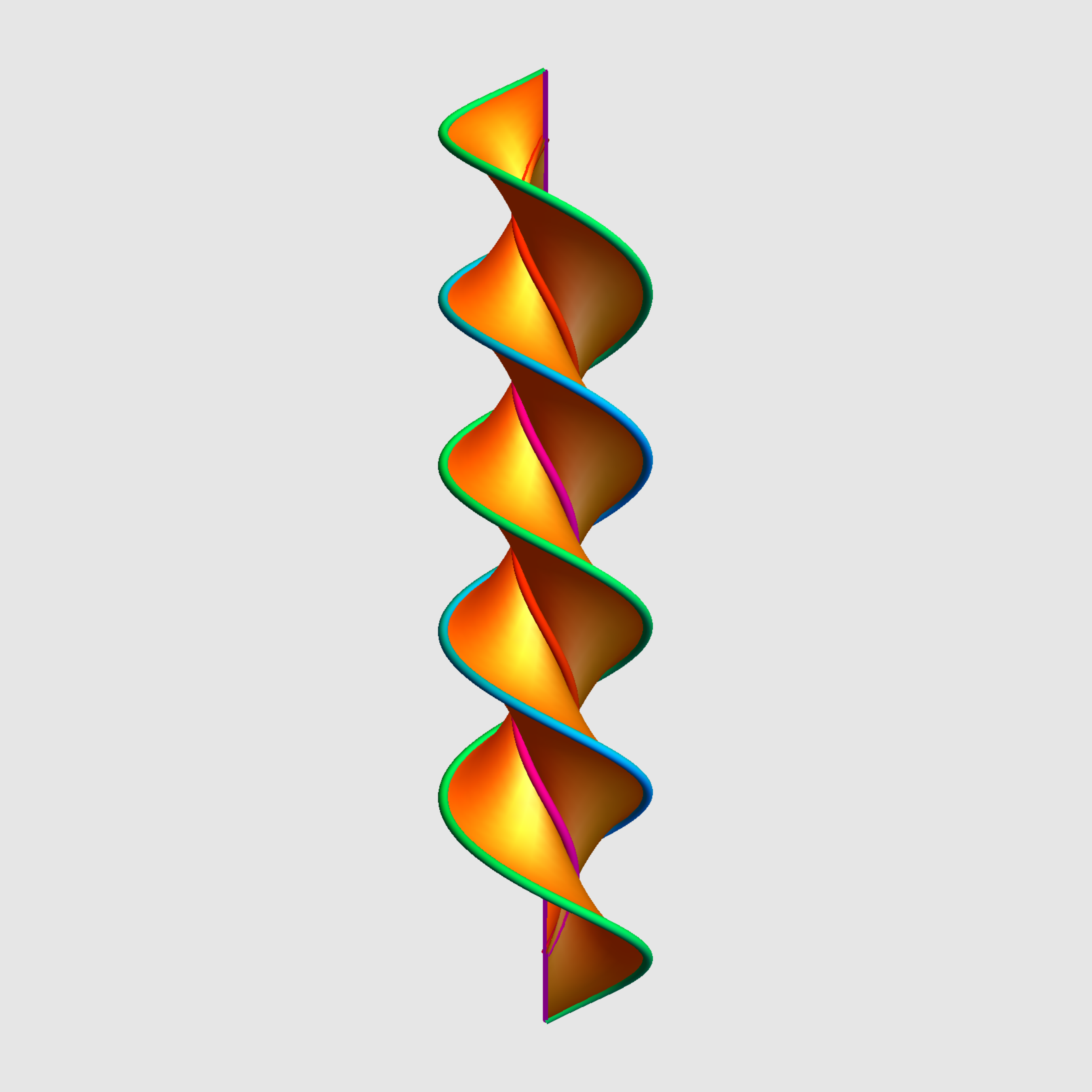}
\caption{Magnetic (left) and electric (right) pseudospherical twisted columns with $\mathfrak{n}=2$
and $\mathfrak{d}=\phi$, $\phi = \text{golden ratio}$.}\label{FIG10}
\end{center}
\end{figure}

\begin{figure}
\begin{center}
\includegraphics[height=6cm,width=6cm]{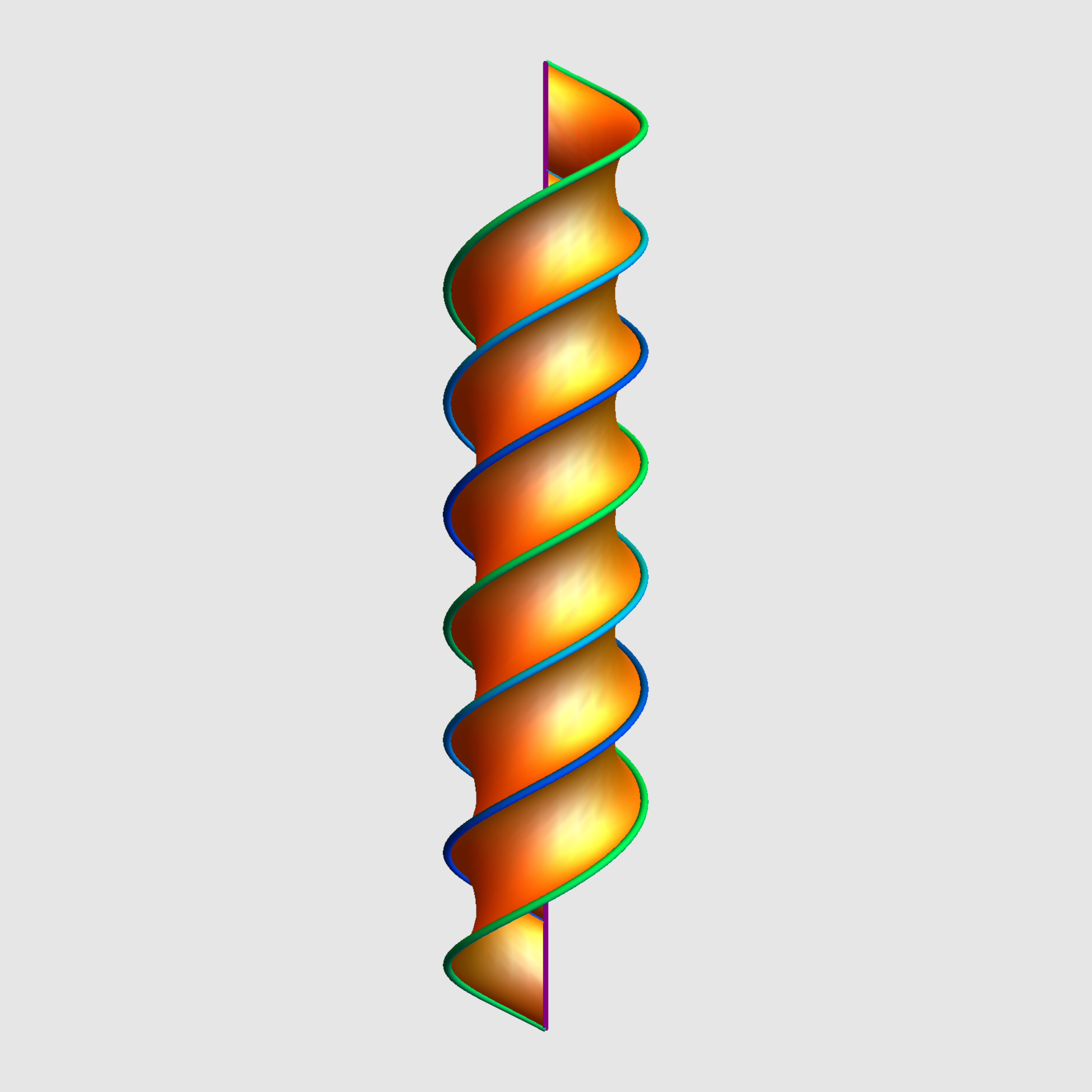}
\includegraphics[height=6cm,width=6cm]{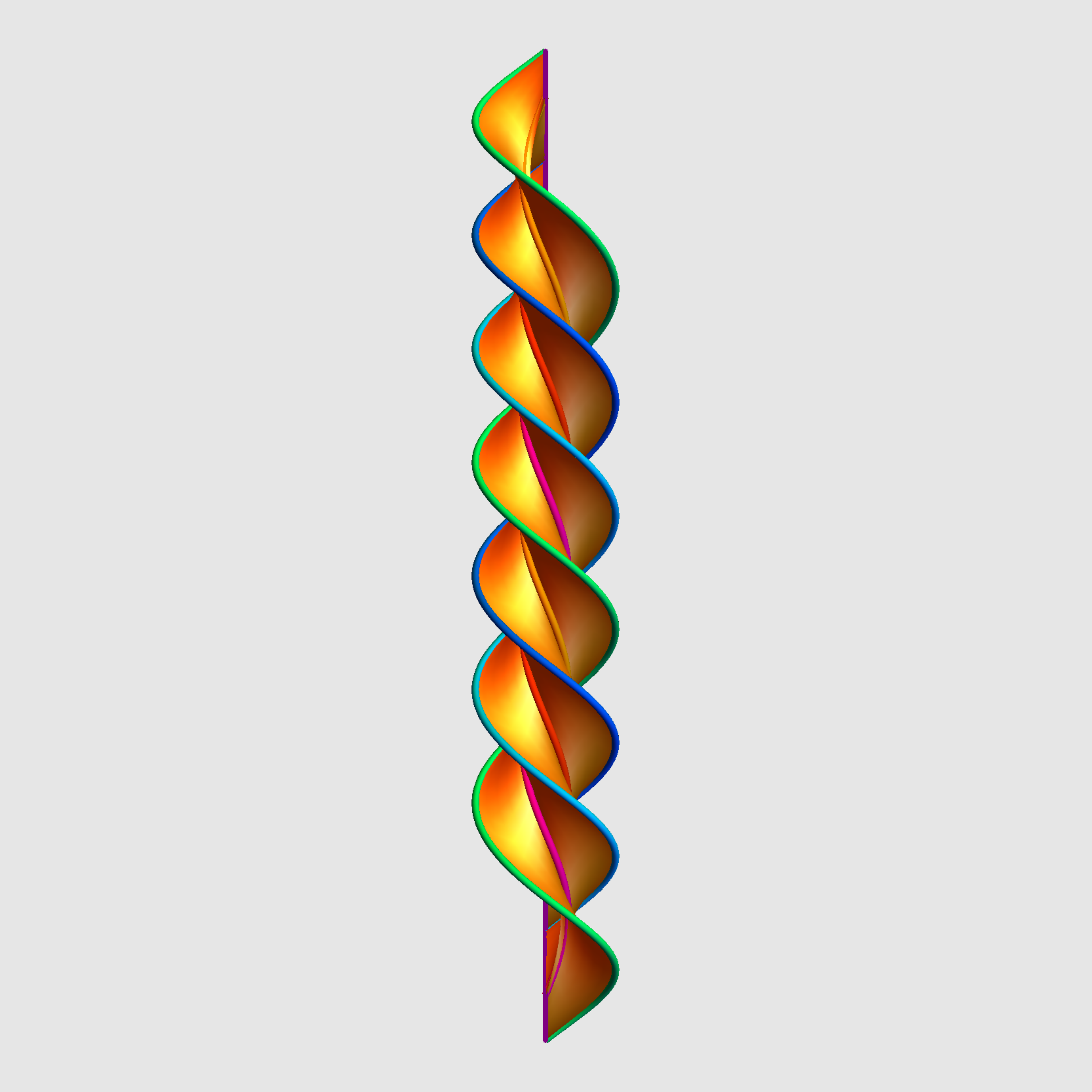}
\caption{Magnetic (left) and electric (right) pseudospherical twisted columns with $\mathfrak{n}=3$
and $\mathfrak{d} =\phi$, $\phi = \text{golden ratio}$.}\label{FIG11}
\end{center}
\end{figure}



\bibliographystyle{amsalpha}

\end{document}